\documentclass[11pt,leqno]{amsart}

\usepackage{amsmath,amsthm, amssymb, latexsym, url,amsthm,amsfonts}
\usepackage{hyperref}

\usepackage{float}

\usepackage{graphicx}
\DeclareGraphicsExtensions{.eps, .pdf,.jpg,.png}
\DeclareGraphicsExtensions{.eps, .pdf,.jpg,.png}

\graphicspath{{Graphics/}}

\usepackage[margin=1in]{geometry}

\newtheorem{thm}{Theorem}[section]
\newtheorem{prop}[thm]{Proposition}
\newtheorem{cor}[thm]{Corollary}

\newtheorem{conj}[thm]{Conjecture}

\theoremstyle{definition}
\newtheorem{defn}[thm]{Definition}

\theoremstyle{remark}

\newtheorem{remarks}[thm]{Remarks}
\newtheorem{remark}[thm]{Remark}

\numberwithin{equation}{section}

\newcommand{\ZZ}{\mathbb{Z}}
\newcommand{\NN}{\mathbb{N}}

\newcommand{\tv}{\operatorname{d_{TV}}}

\begin{document}


\title{
Leading Digits of Mersenne Numbers
}

\author[Cai]{Zhaodong Cai}
\email[Zhaodong Cai]{zhcai@sas.upenn.edu}

\author[Faust]{Matthew Faust}
\email[Matthew Faust]{mhfaust2@illinois.edu}

\author[Hildebrand]{A.J. Hildebrand}
\email[A.J. Hildebrand (corresponding author)]{ajh@illinois.edu}

\author[Li]{Junxian Li}
\email[Junxian Li]{jli135@illinois.edu}

\author[Zhang]{Yuan Zhang}
\email[Yuan Zhang]{yz3yq@virginia.edu}

\address{Department of Mathematics\\
University of Illinois\\
Urbana, IL 61801\\
USA}

\thanks{This work is based on a research project
carried out at the \emph{Illinois Geometry Lab} in Spring 2016 and Fall
2016.  The experimental results in this paper were generated using the
\emph{Illinois Campus Computing Cluster}, a high performance
computing platform at the University of Illinois.}


\date{10/20/2018}

\maketitle

\begin{abstract}
It has long been known that sequences such as the powers of $2$
and the factorials satisfy Benford's Law; that is, leading digits in
these sequences occur with frequencies given by $P(d)=\log_{10}(1+1/d)$,
$d=1,2,\dots,9$. 
In this paper, we consider the leading digits of the Mersenne numbers
$M_n=2^{p_n}-1$, where $p_n$ is the $n$-th prime. In light of known
irregularities in the distribution of primes, one might expect that the
leading digit sequence of $\{M_n\}$ has \emph{worse} distribution
properties than ``smooth'' sequences with similar rates of growth, such as
$\{2^{n\log n}\}$. Surprisingly, the opposite seems to be the true;
indeed, we present data, based on the first billion terms of the sequence
$\{M_n\}$, showing that leading digits of Mersenne numbers behave in many
respects \emph{more regularly} than those in the above smooth sequences.
We state several conjectures to this effect,  and we provide an heuristic
explanation for the observed phenomena based on classic models for the  
distribution of primes.
\end{abstract}







\section{Introduction}
\label{sec:intro}

\subsection{Benford's Law}
If the leading digits (in base $10$) of the sequence
$2$, $4$, $8$, $16$, $32$, $64$, $128$, $256$, $512$, $1024$, ...
of powers of $2$ are tabulated, one finds that
the digit $1$ occurs around $30.1\%$ of the time, the digit $2$ occurs around
$17.6\%$ of the time, while the digit $9$ occurs only around $4.6\%$ of the
time.  This is an instance of \emph{Benford's Law}, an empirical ``law''
that says that leading digits in many real-world and mathematical
data sets tend to follow the \emph{Benford distribution}, 
depicted in Figure \ref{fig:benford-distribution}, and given by
\begin{equation}
\label{eq:benford}
P(d)
=\log_{10}\left(1+\frac1d\right),\quad d=1,2,\dots,9.
\end{equation}

\begin{figure}[H]
\begin{center}
\includegraphics[width=.5\textwidth]{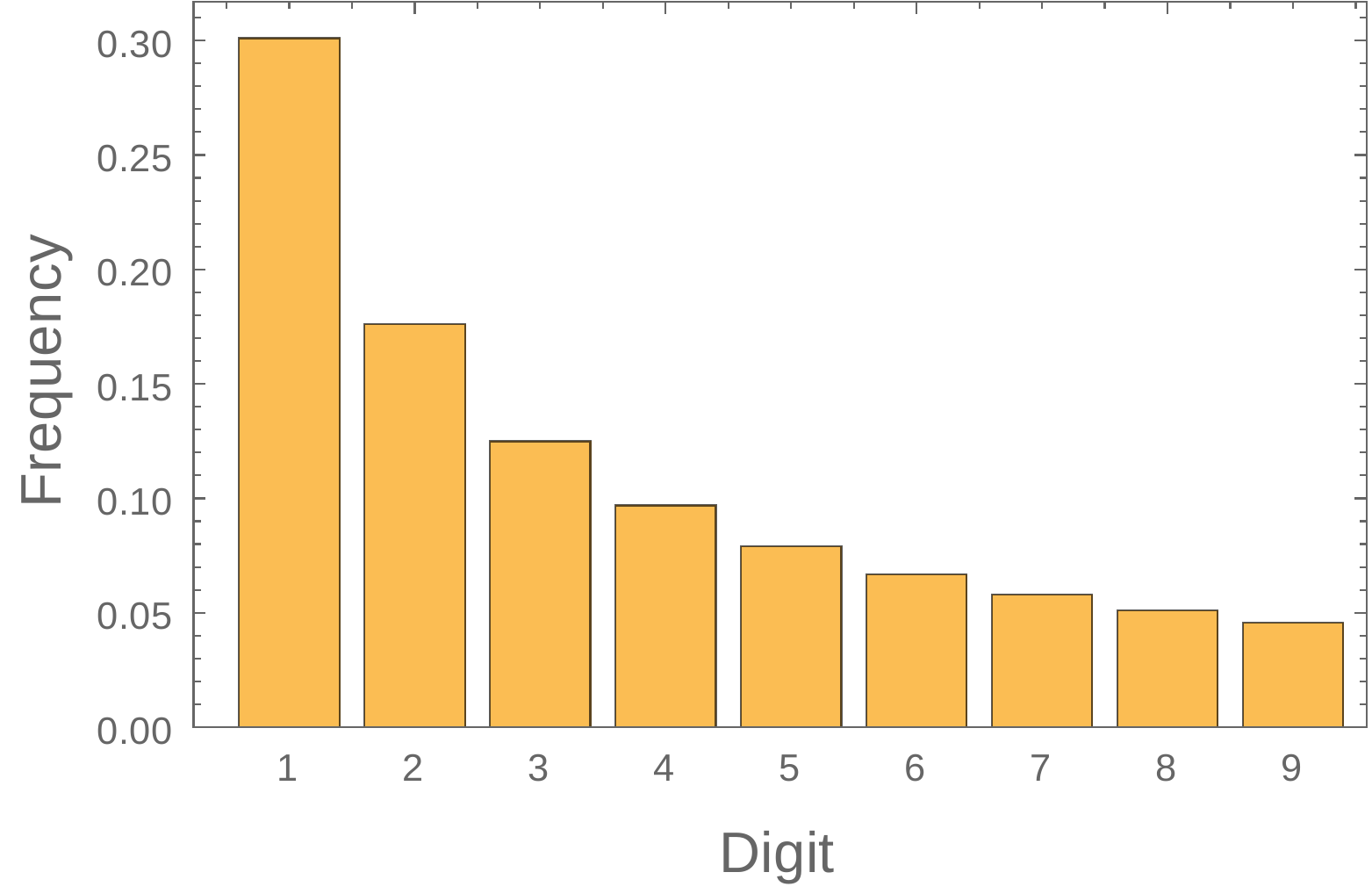}
\end{center}
\caption{The Benford distribution, $P(d)=\log_{10}(1+1/d)$.}
\label{fig:benford-distribution}
\end{figure}

The peculiar first-digit distribution  given by \eqref{eq:benford} 
is named after Frank Benford \cite{benford1938}, who in 1938 compiled extensive
empirical evidence for the ubiquity of this distribution across a wide
range of real-life data sets, though it had been observed some
fifty years earlier by the astronomer Simon Newcomb \cite{newcomb1881}.
In recent decades, Benford's Law has
received renewed interest, in part  because of its applications as a tool
in fraud detection.  For general background on Benford's Law and its
applications we refer to the articles by 
Hill \cite{hill1995} and 
Raimi \cite{raimi1976},
the in-depth survey by Berger and Hill \cite{berger2011},
and the recent books by  
Berger and Hill \cite{berger2015}, Miller \cite{miller2015}, 
and Nigrini \cite{nigrini2012}. 
Additional references can be found in the online bibliographies 
\cite{beebe2017}, \cite{benfordonline}, 
and \cite{hurliman2006}.

\subsection{Benford's Law in mathematics}
From a mathematical point of view, Benford's Law is closely connected
with the theory of \emph{uniform distribution modulo $1$} \cite{kuipers1974}.
In 1977 Diaconis \cite{diaconis1977} used this connection to prove rigorously
that Benford's Law holds  for
a class of exponentially growing sequences which includes the powers of $2$
and the sequence of
factorials.  That is, each of these sequences $\{a_n\}$ satisfies
\begin{equation}
\label{eq:benford-limit}
\lim_{N\to\infty}\frac1N
\#\left\{n\le N: \text{$a_n$ has leading digit $d$}\right\}
=\log_{10}\left(1+\frac1d\right),\quad d=1,2,\dots,9.
\end{equation}
Table \ref{table:benford-2n} illustrates this result for the sequence
of powers of $2$. The agreement between actual leading digit counts and
the expected counts based on the Benford frequencies \eqref{eq:benford}
is uncannily good: The Benford predictions are within $\pm10$ of the
actual counts among the first billion terms of the sequence.\footnote{%
In general, for sequences of the form $\{a^n\}$ the quality of the
agreement between the actual and predicted leading digit counts is
related to diophantine approximation properties of the number $\log_{10}
a$ (see Proposition \ref{prop:benford-ud-mod1} below and \cite[Chapter 2,
Theorem 3.2]{kuipers1974}), but it is also affected by any linear relations
between the numbers $\log_{10}a$ and $\log_{10}d$, $d=1,2,\dots,9$. For 
the sequence $\{2^n\}$, the latter aspect comes into play and accounts
in part for the small errors observed in Table \ref{table:benford-2n}; see 
\cite{benford-error}.} 


\begin{table}[H]
\begin{center}
\begin{tabular}{|c|r|r|r|}
\hline
Digit& Count & Benford Prediction & Error
\\
\hline
\hline
1 & 301029995 & 301029995.66 & -0.66
\\
\hline
2 & 176091267 & 176091259.06 & 7.94
\\
\hline
3 & 124938729 & 124938736.61 & -7.61
\\
\hline
4 & 96910014 & 96910013.01 & 0.99
\\
\hline
5 & 79181253 & 79181246.05 & 6.95
\\
\hline
6 & 66946788 & 66946789.63 & -1.63
\\
\hline
7 & 57991941 & 57991946.98 & -5.98
\\
\hline
8 & 51152528 & 51152522.45 & 5.55
\\
\hline
9 & 45757485 & 45757490.56 & -5.56
\\
\hline
\end{tabular}
\end{center}
\caption{%
Actual versus predicted counts of leading digits among the first $10^9$ terms
of the sequence $\{2^n\}$. The predicted counts are given by 
$NP(d)$, where $N=10^9$ is the number of terms, $d$ is the
digit, and $P(d)=\log_{10}(1+1/d)$ is the Benford frequency for digit $d$,
given by \eqref{eq:benford}.
}
\label{table:benford-2n}
\end{table}


More recently,
H\"urliman \cite{hurliman2009} investigated Benford's Law for a variety of
classical arithmetic sequences and special numbers such as the Catalan
numbers.  Mass\'e and Schneider \cite{masse2015}
established Benford's Law for a large class of arithmetic sequences defined
by growth conditions. In particular, they showed that Benford's Law holds 
for sequences of the form $\lambda n^{R(n)}e^{S(n)}$, where $\lambda>0$ and
$R(n)$ and $S(n)$ are polynomials satisfying some mild conditions. 
Examples covered by their results include the sequences $\{2^{n^h}\}$,
where $h$ is a fixed positive integer, and $\{n^{n^\alpha}\}$, where
$\alpha>0$.

Table \ref{table:benford-smooth-sequences} gives a numerical illustration
of these results, based on leading digit data for the first $10^9$ terms of
the sequences $\{2^{n^2}\}$, $\{2^{n\log n}\}$, and $\{n^n\}$.
In all three cases, the deviation between the predicted and actual counts
of leading digits is in the order of $10^4$. 
While not nearly as small as the errors for the sequence $\{2^n\}$, these
deviations are comparable to the  squareroot type deviation one would
expect for a random sequence.


\begin{table}[H]
\begin{center}
\begin{tabular}{|c|r|r|r|r|}
\hline
Digit& Benford Prediction & 
$\{2^{n^2}\}$  &
$\{2^{n\log n}\}$  &
 $\{n^n\}$  
\\
\hline
\hline
1 & 301029995.66
& 2954.34 
&16567.34
& 6820.34  
\\
\hline
2 & 176091259.06
& -6673.06
&-11543.06
& -10500.06  
\\
\hline
3 & 124938736.61
& 121.39 
& 16785.39
& -17051.61  
\\
\hline
4 & 96910013.01
& -59.01 
& 3205.99
& -4763.01  
\\
\hline
5 & 79181246.05
&  7.95
& -16409.05
& 20660.95  
\\
\hline
6 & 66946789.63
& 8897.37
& 6000.37
& -3421.63  
\\
\hline
7 & 57991946.98
& -3733.98
& -1566.98
& 21179.02  
\\
\hline
8 & 51152522.45
& 236.55
& -9687.45
& -6807.45  
\\
\hline
9 & 45757490.56
& -1751.56  
& -3352.56
& -6116.56  
\\
\hline
\end{tabular}
\end{center}
\caption{
Deviations from predicted counts for leading digits among the first $10^9$
terms of the sequences $\{2^{n^2}\}$, $\{2^{n\log n}\}$, $\{n^n\}$. 
}
\label{table:benford-smooth-sequences}
\end{table}


\subsection{Limitations of Benford's Law}
Sequences of polynomial or slower rate of growth such as the sequence of
squares do not satisfy Benford's Law in the above asymptotic density sense,
though in many cases Benford's Law can be shown
to hold in some weaker form, for example, with the natural asymptotic
density replaced by other notions of density; see Mass\'e and Schneider
\cite{masse2011} for a survey.

The failure of Benford's Law for sequences of polynomial growth is due to
the fact that the leading digits of such sequences stay constant over long
enough intervals to prevent the asymptotic relation
\eqref{eq:benford-limit} from taking hold.  For example, $n^2$ has leading
digit $1$ whenever $n$ falls into an interval of the form
$[10^k,\sqrt{2}\cdot 10^k)$, $k=1,2,\dots$. If \eqref{eq:benford-limit} 
were to hold, then only a fraction $\log_{10}2\approx 0.301$ of these terms
would have leading digit $1$.

In recent work \cite{local-benford} we exhibited 
another limitation to Benford's Law for arithmetic sequences: 
Namely, exponentially growing sequences such as those in Table
\ref{table:benford-smooth-sequences}
tend to have very poor \emph{local} Benford distribution properties, even
though, from a \emph{global} point of view, they provide an excellent match
to Benford's Law. For example, for the sequence $\{2^{n^h}\}$, where 
$h$ is a positive integer, $k$-tuples of leading digits of consecutive
terms in the sequence do behave ``independently'' when $k\le h$, but not
when $k>h$.

\subsection{Benford's Law for arithmetic sequences}
The current state of knowledge on the validity of Benford's Law for
``smooth'' arithmetic sequences can be summarized as follows:

\begin{itemize}
\item[(I)] \textbf{Sequences such as the squares that grow
at linear or polynomial rate.} Benford's Law does not hold in the usual
asymptotic density sense, though it may hold with respect to other density
notions such as analytic or logarithmic densities \cite{masse2011}.

\item[(II)] \textbf{Sequences such as $\{2^n\}$, $\{2^{n^2}\}$, or 
$\{n^n\}$ that grow at 
faster than polynomial rate, but whose logarithms grow at polynomial rate.}
Large classes of such sequences have been shown to satisfy Benford's Law
\cite{masse2015}. On the other hand, as shown in \cite{local-benford},
such sequences tend to have poor \emph{local} distribution properties,
with the quality of the local fit to Benford's Law being closely tied 
to the rate of growth of the sequence: Faster growing sequences generally
are better behaved at the local level with respect to Benford's Law.

\item[(III)] \textbf{Sequences whose logarithms grow at faster than
polynomial rate.} Extrapolating from the results for the case of polynomial
growth, one may expect such sequences to \emph{generally} satisfy Benford's
Law, both at the global and the local level, in the sense that the associated
leading digit sequence behaves like a sequence of independent
Benford-distributed random variables. This can indeed be shown to be the
case for ``almost all'' doubly exponential sequences \cite{local-benford},
though proofs of Benford's Law for \emph{specific} sequences of doubly
exponential growth such as $\{2^{2^n}\}$ remain elusive; see the
remark at the end of  \cite{masse2015}.  
\end{itemize}

\subsection{Sequences involving prime numbers}
The above-mentioned results focus on the leading digit behavior of 
``smooth'' sequences, i.e., sequences of the form $\{f(n)\}$, where
$f(x)$ is some well-behaved function of $x$.  One can ask similar
questions about sequences that are defined in
terms of prime numbers.  The sequence of prime numbers $\{p_n\}$ 
itself does not satisfy Benford's Law for the same reason that
polynomial sequences do not satisfy this law:  Since $p_n\sim n\log n$
as $n\to\infty$ (see, for example, Theorem 4.1 in \cite{apostol1976}), 
the rate of growth of $\{p_n\}$ is too slow for the asymptotic relation
\eqref{eq:benford-limit} to take hold.  However, a number of authors
have shown that the primes satisfy various weaker forms of this law; see
Whitney \cite{whitney1972}, Schatte \cite{schatte1983}, 
Cohen and Katz \cite{cohen-katz1984},
Fuchs and Letta \cite{fuchs-letta1996}, Luque and Lacasa
\cite{luque-lacasa2009}, 
Eliahou et al. \cite{eliahou2013}, 
and Mass\'e and Schneider \cite{masse2011}. 

In light of the above heuristic, 
it is reasonable to expect that Benford's Law holds
for sufficiently fast growing sequences defined in terms of prime numbers.
Mass\'e and Schneider \cite{masse2014} showed that this is indeed the case  
for the sequence $\{P_n\}$  of \emph{primorial numbers} defined by 
$P_n=\prod_{k=1}^n p_k$.

\subsection{The Mersenne numbers}
In this paper we consider another classic sequence involving prime numbers,
the \emph{Mersenne numbers}, defined as\footnote{We emphasize that we do 
\emph{not} require
$M_n$ to be prime, but we do require the exponent, $p_n$, to be prime.  In
other words, the sequence $\{M_n\}$ is the sequence of candidates for
Mersenne primes.}
\begin{equation}
\label{eq:mersenne-def}
M_n=2^{p_n}-1.
\end{equation}
The first twenty terms of this sequence are given in Table
\ref{table:mersenne-numbers}.


\begin{table}[H]
\begin{center}
\begin{tabular}{|c|r|r|}
\hline
$n$ & $p_n$ & $M_n=2^{p_n}-1$  
\\
\hline
\hline
1&  2&  3
\\ \hline
2&  3&  7
\\ \hline
3&  5&  31
\\ \hline
4&  7&  127
\\ \hline
5&  11&  2047
\\ \hline
6&  13&  8191
\\ \hline
7&  17&  131071
\\ \hline
8&  19&  524287
\\ \hline
9&  23&  8388607
\\ \hline
10&  29&  536870911
\\
\hline
\end{tabular}
\hspace{1em}
\begin{tabular}{|c|r|r|}
\hline
$n$ & $p_n$ & $M_n=2^{p_n}-1$  
\\
\hline
\hline
11& 31& 2147483647
\\ \hline
12& 37& 137438953471
\\ \hline
13& 41& 2199023255551
\\ \hline
14& 43& 8796093022207
\\ \hline
15& 47& 140737488355327
\\ \hline
16& 53& 9007199254740991
\\ \hline
17& 59& 576460752303423487
\\ \hline
18& 61& 2305843009213693951
\\ \hline
19& 67& 147573952589676412927
\\ \hline
20& 71& 2361183241434822606847
\\
\hline
\end{tabular}
\end{center}
\caption{%
The first $20$ Mersenne numbers, $M_n=2^{p_n}-1$.
}
\label{table:mersenne-numbers}
\end{table}

Since $p_n\sim n\log n$, the sequence $\{M_n\}$ has a rate of growth
between that of the sequences $\{2^n\}$ and $\{2^{n^2}\}$, and very
similar to that of the sequence $\{2^{n\log n}\}$.  In terms of the above
hierarchy, it is a sequence of type (II).  Thus, one might expect the
sequence $\{M_n\}$ to have excellent global, but poor local distribution
properties with respect to Benford's Law. 

From a \emph{global} point of view, the behavior is indeed as expected. 
We show that the sequence $\{M_n\}$ satisfies Benford's Law
and we provide numerical evidence suggesting that the quality of the fit is
comparable to that of other sequences of similar rate of growth.

On the other hand, the \emph{local} distribution of leading digits of
$\{M_n\}$ is completely different from that of other sequences of type
(II), and more like that of sequences of type (III).  An illustration
of these differences is given in Figure
\ref{fig:benford-waits-overview}, which shows the distribution of
``waiting times'' between successive occurrences of $1$ as leading
digit for the sequences $\{M_n\}$, $\{2^n\}$, $\{2^{n\log n}\}$, and
$\{2^{n^2}\}$. Only the Mersenne sequence, $\{M_n\}$, exhibits the
geometric waiting time distribution that one would expect for a random
sequence of Benford-distributed digits.  The other three sequences have
distinctly different waiting time distributions.\footnote{That the digit
$1$ waiting times for the sequence $\{2^n\}$ are either $3$ or $4$ is easy
to see by tracking leading digits after successive multiplications by $2$;
for the other three sequences, however, the set of possible values 
of the waiting times seems to have a much more complicated structure.}

\begin{figure}[H]
\begin{center}
\includegraphics[width=.45\textwidth]{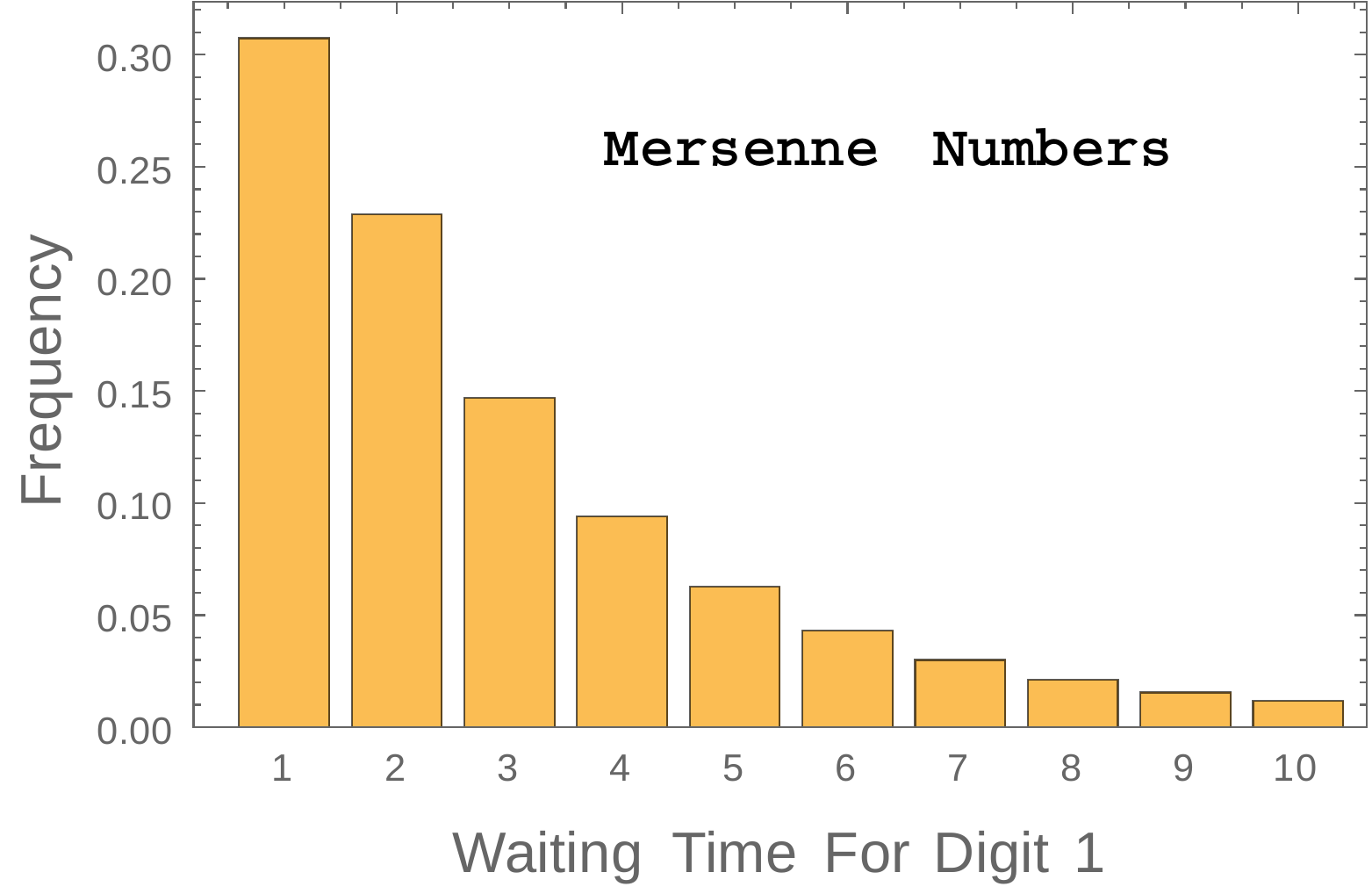}
\hspace{1em}
\includegraphics[width=.45\textwidth]{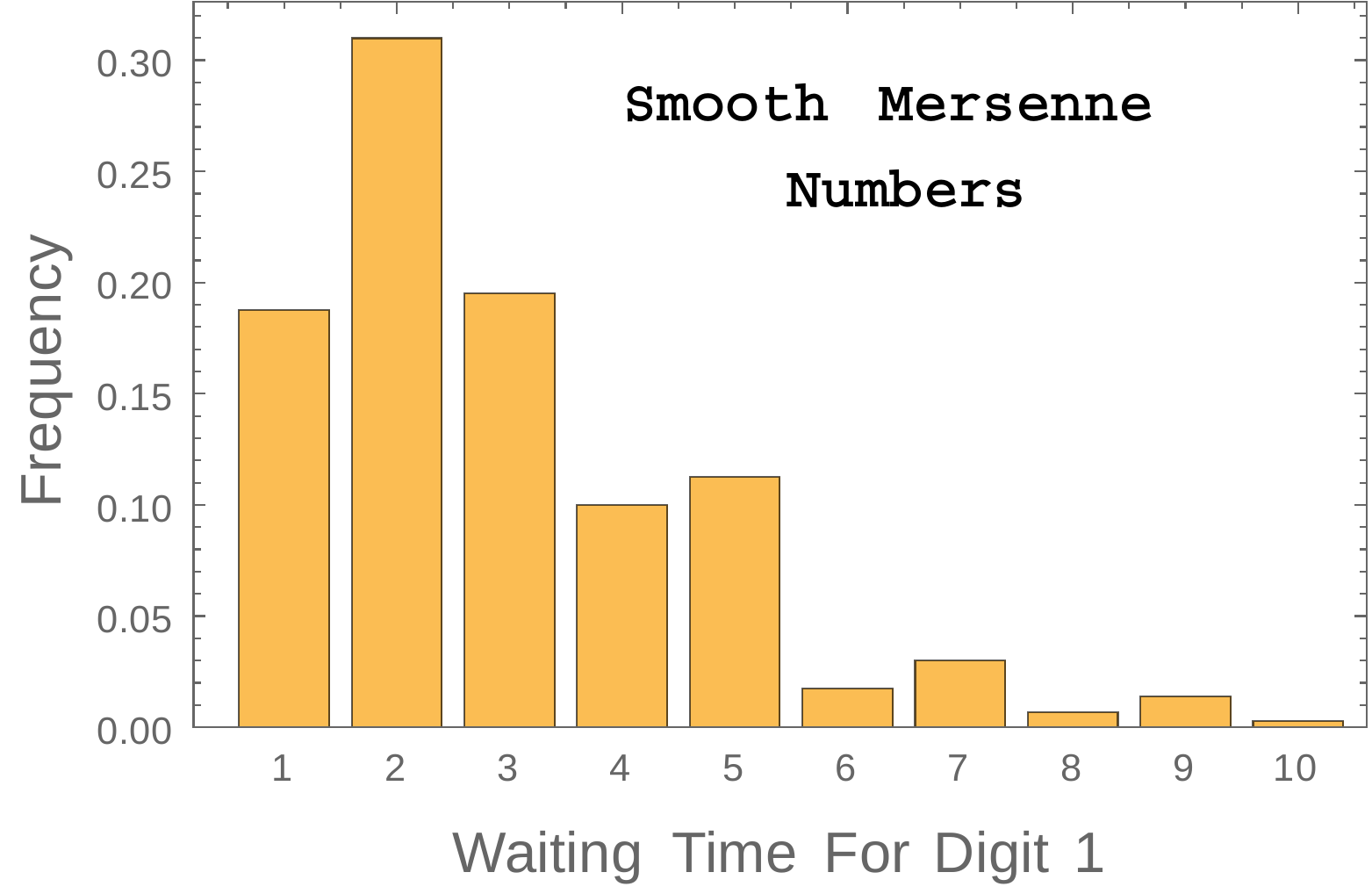}
\\[1ex]
\includegraphics[width=.45\textwidth]{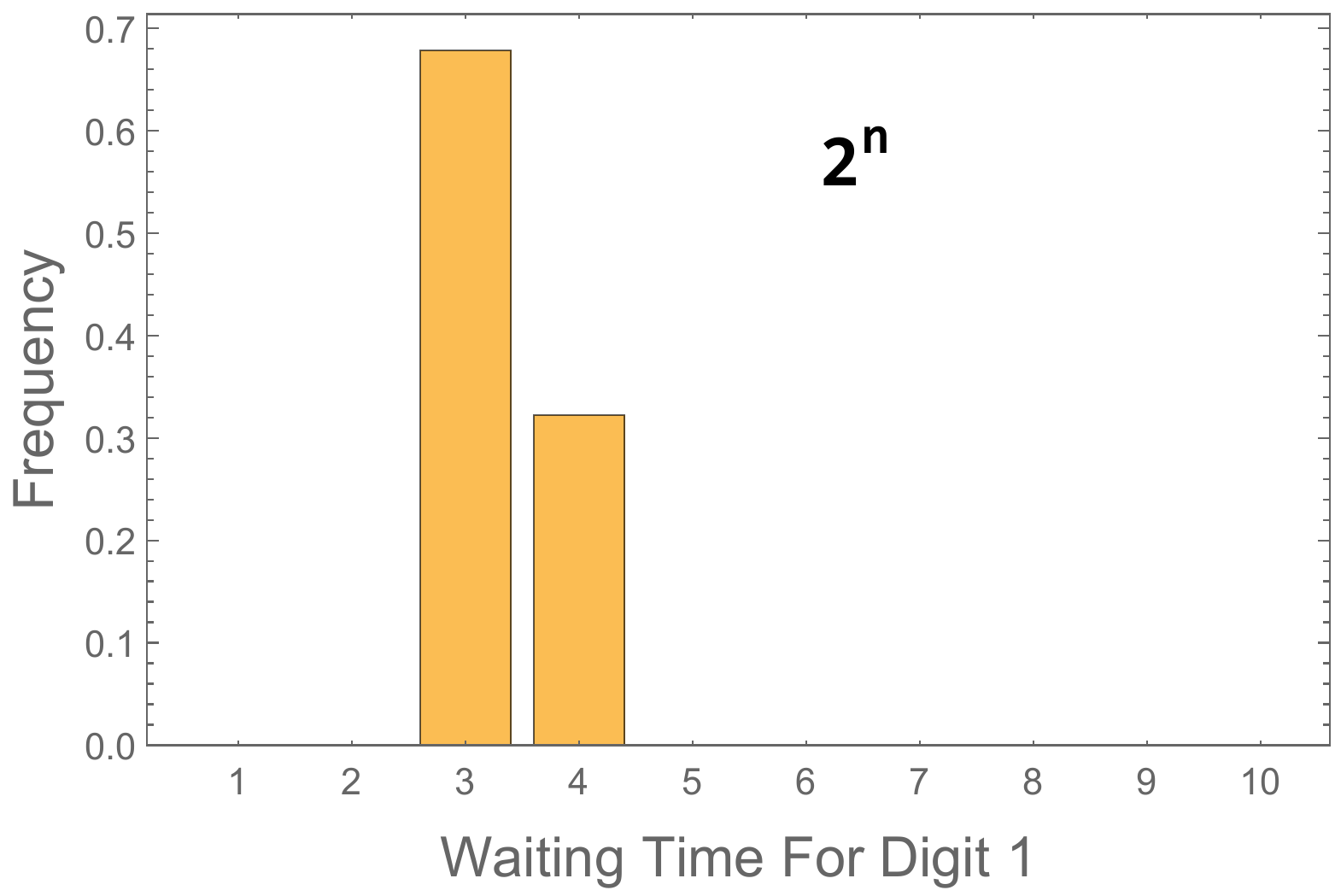}
\hspace{1em}
\includegraphics[width=.45\textwidth]{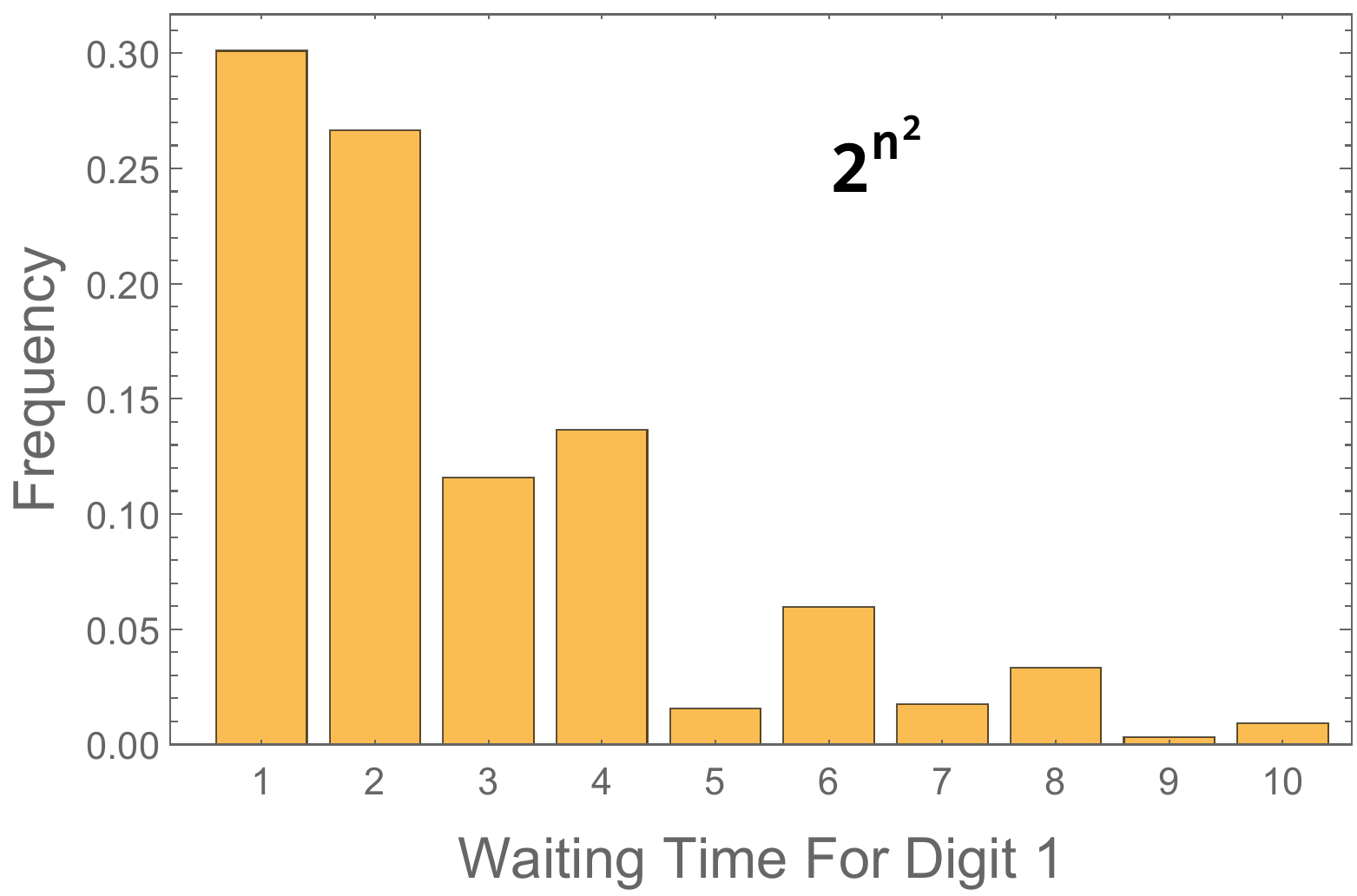}
\end{center}
\caption{The distribution of ``waiting times'' between occurrences of 
leading digit $1$
for the sequence
of Mersenne numbers $M_n=2^{p_n}-1$ (top left chart), 
and for three ``smooth'' sequences with similar rates of growth: 
$\{2^{n\log n}\}$ (top right), $\{2^n\}$ (bottom left), 
and $\{2^{n^2}\}$ (bottom right). Of these four sequences only the
Mersenne sequence exhibits a geometric waiting time
distribution.
}
\label{fig:benford-waits-overview}
\end{figure}

This surprising discrepancy between the \emph{local} Benford distribution
properties of the sequence $\{M_n\}$ and similar ``smooth'' sequences is the
main finding of this paper.  We conjecture that, in contrast to 
smooth sequences such as those in Tables \ref{table:benford-2n} and 
\ref{table:benford-smooth-sequences},
the sequence of leading digits of 
$\{M_n\}$ behaves like a sequence of independent Benford-distributed random
variables.  We provide numerical evidence in support of this
conclusion, and we give an heuristic explanation for the apparent
discrepancy in the behaviors of $\{M_n\}$ and similar smooth sequences.

\subsection{Outline of paper}

The remainder of this paper is organized as follows: In Section
\ref{sec:summary}  we state
our main conjectures and results. In Section \ref{sec:data}
we describe the numerical
data on which these conjectures are based and the approach we have taken 
to generate the data. 
In Section \ref{sec:Mn-global}
we show that $\{M_n\}$ satisfies Benford's Law,
and we provide numerical data on the quality of the fit. 
In Section \ref{sec:Mn-local}  we present experimental
data supporting our conjectures on the local distribution of leading digits
of $\{M_n\}$. In Section \ref{sec:smooth-local}
we show that smooth sequences with similar rates of
growth do not satisfy these conjectures.
Section \ref{sec:conclusion}
contains a summary of our findings, along with some
remarks and open questions.


\section{Summary of Results and Conjectures}
\label{sec:summary}

\subsection{Notations and definitions}
Given a positive real number $x$, we denote by $D(x)$ the leading (i.e.,
most significant) digit of $x$ in base $10$.   More precisely, we define
$D(x)$ by 
\begin{equation}
\label{eq:D(x)}
D(x)=d\Longleftrightarrow d\cdot 10^{k}\le x < (d+1)10^{k}\quad
\text{for some $k\in\ZZ$}
\end{equation}
for $d\in\{1,2,\dots,9\}$.  Note that this definition does not
require $x$ to be an integer; for example, we
have $D(\pi)=3$ and $D(0.0314)=3$.
We let $P(d) =\log_{10}(1+1/d)$ denote the Benford frequency for digit
$d$, as defined in \eqref{eq:benford}.

\begin{defn}[Global Benford Distribution]
\label{def:benford}
A sequence $\{a_n\}$ of positive real numbers
is said to be \emph{Benford distributed} (or, equivalently, said to
satisfy \emph{Benford's Law})  if 
\begin{equation}
\label{eq:benford3}
\lim_{N\to\infty} \frac1N\#\{n\le N: D(a_n)=d\}=P(d)\quad
\text{for $d=1,2,\dots,9$.}
\end{equation}
\end{defn}

\begin{defn}[Local Benford Distribution]
\label{def:local-benford}
Let $k$ be a positive integer. 
A sequence $\{a_n\}$ of positive real numbers 
is called \emph{locally Benford distributed of
order $k$} if 
\begin{align}
\label{eq:local-benford}
\lim_{N\to\infty} &\frac1N\#\{n\le N: D(a_{n+i})=d_i\quad 
(i=0,1,\dots,k-1)\}
\\
\notag
&\qquad
= P(d_0)P(d_1)\dots P(d_{k-1})
\quad \text{for $d_i= 1,\dots,9$ ($i=0,1,\dots,k-1$).} 
\end{align}
\end{defn}

\begin{remarks}
\mbox{}
(1) Definition \ref{def:benford} is one of several common definitions 
of Benford's Law used in the literature. We chose this particular version
over others in the
literature because of its simplicity and intuitiveness.

(2) The case $k=1$ in \eqref{eq:local-benford} reduces to the definition
\eqref{eq:benford3} of a (global) Benford distributed sequence.  
It is immediate from the definition that 
a sequence that is locally Benford distributed of order $k$
is also locally Benford distributed of any order $k'\le k$.
Thus, the concept of local Benford distribution of a sequence refines that
of Benford distribution and establishes a hierarchy of classes of
sequences with successively stronger local distribution properties.

(3) The above definitions can be extended in a natural way 
to other bases, and we expect that most of our results and conjectures
remain valid for such generalized versions of Benford's Law. We decided to
focus on the standard case of base $10$ in order to avoid unnecessary
notational complications. All of the features we expect to hold in the
general case are already present in base $10$.

\end{remarks}

Given a sequence $\{a_n\}$ of positive real numbers
and $d\in\{1,2,\dots,9\}$, we define sequences   $\{n_i^{(d)}\}$ by
\begin{align}
\label{eq:n_i}
\{n_1^{(d)}<n_2^{(d)}<n_3^{(d)}<\cdots  \} &= \{n\in\NN: D(a_n)=d\},
\end{align}
and we let 
\begin{align}
\label{eq:w_i}
w_i{(d)}&=n_{i+1}^{(d)}-n_i^{(d)}.
\end{align}
In other words, the numbers $n_i^{(d)}$ are the successive indices $n$ at
which $a_n$ has leading digit $d$, and the numbers 
$w_i^{(d)}$ are the ``waiting times'', or gaps,
between these occurrences. 

If $\{a_n\}$ is Benford distributed, then the numbers 
$n_i^{(d)}$ occur with asymptotic frequency $P(d)$, so the average gap
between these numbers is $1/P(d)$, i.e., we have
\begin{equation}
\label{eq:w-average}
\lim_{N\to\infty} \frac1N\sum_{i\le N} w_i(d) = \frac{1}{P(d)}
\quad\text{for $d=1,2,\dots,9$.} 
\end{equation}
In fact, it is not hard to see that the converse is also true.  That is,
\eqref{eq:w-average} holds if and only if $\{a_n\}$ is Benford
distributed.

In general, the average statement \eqref{eq:w-average} is all we can
say about the waiting times of a sequence that is Benford distributed. 
However, 
for sequences that are locally ``well behaved'' we expect more to be true: 
Namely, we expect the waiting times between these occurrences to have
geometric distribution with mean $1/P(d)$. 
We thus make the following definition.

\begin{defn}[Benford Distributed Waiting Times]
\label{def:benford-waits}
A sequence $\{a_n\}$ of positive real numbers 
is said to have \emph{Benford distributed waiting times} if
\begin{align}
\label{eq:w-distribution}
\lim_{N\to\infty} \frac1N\#\left\{i\le N: w_i(d)=k\right\}
&= P(d)(1-P(d))^{k-1}
\\
\notag
&\qquad\text{for $d=1,2,\dots,9$ and $k=1,2,\dots$.} 
\end{align}
\end{defn}

It is not hard to see that a sequence that is locally Benford distributed
of any order $k$ has Benford distributed waiting times.
Thus, we have the chain of implications:
\begin{gather*}
\text{Locally Benford distributed of any order $k$,
\eqref{eq:local-benford}}
\\
\Downarrow
\\
\text{Benford distributed waiting times, \eqref{eq:w-distribution}}
\\
\Downarrow
\\
\text{Average waiting time property, \eqref{eq:w-average}}
\\
\Updownarrow
\\
\text{Benford's Law, \eqref{eq:benford3}}.
\end{gather*}

\subsection{Global distribution properties}  
Recall the definition of the Mersenne numbers: 
\begin{equation*}
M_n=2^{p_n}-1,
\end{equation*}
where $p_n$ denotes the $n$-th prime. 

\begin{thm}[Benford Law for $\{M_n\}$]
\label{thm:Mn-global}
The sequence $\{M_n\}$ is Benford-distributed
i.e., satisfies \eqref{eq:benford3}.
\end{thm}

This result is a consequence of a theorem of Vinogradov
\cite{vinogradov1947}  and may be known to experts in the field, but we
were unable to find a specific reference in the literature. We will supply
a proof in Section \ref{sec:Mn-global}.

In light of this result, it is natural to consider the size of the error
in the Benford approximation for the frequencies of leading digits of 
$\{M_n\}$, i.e., the quantities 
\begin{equation}
\label{eq:benford-error}
E_d(N)=\#\{n\le N: D(M_n)=d\}- N P(d).
\end{equation}

As Tables \ref{table:benford-2n} and \ref{table:benford-smooth-sequences} 
show, for smooth sequences the size of this error can vary dramatically, 
from a logarithmic or even bounded error in the case of $\{2^n\}$ 
to the squareroot size oscillations typically associated with random
sequences.  Our data (see Section \ref{sec:Mn-global})
suggests that the sequence $\{M_n\}$ falls into the
latter class.

\begin{conj}[Benford Error for $\{M_n\}$] 
\label{conj:benford-error}
The Benford errors $E_d(N)$ defined by \eqref{eq:benford-error} satisfy
\[
E_d(N)=O(N^{1/2+\epsilon})
\quad\text{and}\quad E_d(N)\not=O(N^{1/2-\epsilon})
\]
for any fixed $\epsilon>0$.
\end{conj}

\subsection{Local distribution properties} 

We now turn to the local distribution properties of the leading digits of
$\{M_n\}$.
We make the following conjectures.

\begin{conj}[Local Benford Distribution of $\{M_n\}$]
\label{conj:local-benford}
The sequence $\{M_n\}$ is locally Benford distributed of any order $k$.
That is, for any positive integer $k$, the 
leading digits of $k$-tuples of consecutive terms in this sequence 
behave like $k$ independent Benford-distributed random variables.  
\end{conj}


\begin{conj}[Benford Waiting Times for $\{M_n\}$]
\label{conj:benford-waits}
The sequence $\{M_n\}$ has Benford-distributed waiting times.
That is, the waiting
times between occurrences of leading digit $d$ behave like 
geometric random variables with parameter $p=P(d)$.
\end{conj}

These conjectures are motivated by the numerical data we will present in
Section \ref{sec:Mn-local} of this paper, and by heuristic arguments,
described in Section \ref{sec:conclusion} and  based on classical 
conjectures about the local distribution of primes.  

\bigskip

In stark contrast to the behavior of $\{M_n\}$ predicted by these
conjectures, the following result shows that smooth sequences with
similar growth rates do not satisfy the conjectures.%

\begin{thm}[Failure of Local Benford Law for Mersenne-like Smooth Sequences]
\label{thm:smooth-local-benford}
Let $\{a_n\}$ be a sequence of positive real numbers such that the
logarithmic differences
\[
\Delta \log a_n = \log a_{n+1}-\log a_n
\]
satisfy 
\begin{equation}
\label{eq:smooth-local-benford-I}
\Delta\log a_n\to\infty \quad (n\to\infty)
\end{equation}
and 
\begin{equation}
\label{eq:smooth-local-benford-II}
\Delta\log a_{n+1} = \Delta\log a_n +O\left(\frac1n\right)  \quad (n\to\infty).
\end{equation}
Then:
\begin{itemize}
\item[(i)] $\{a_n\}$ is not locally Benford distributed of order $k$ when 
$k\ge2$.
\item[(ii)] $\{a_n\}$ does not have Benford distributed waiting times.
\end{itemize}
\end{thm}

The theorem applies to a large class of \emph{smooth} sequences
with growth rates similar to that of the Mersenne numbers. 
In particular, it is easy to check that conditions 
\eqref{eq:smooth-local-benford-I} and \eqref{eq:smooth-local-benford-II} 
hold for the sequences $\{n!\}$, $\{n^n\}$, and  $\{2^{n\log n}\}$.  Thus
we have the following corollary.

\begin{cor}
\label{cor:smooth-local-benford}
The sequences $\{n!\}$, $\{n^n\}$, and $\{2^{n\log n}\}$ are 
not locally Benford distributed of order $2$ (or larger)
and do not have Benford distributed waiting times.
\end{cor}

For more general results of this type see \cite{local-benford}.

Figure \ref{fig:benford-waits-overview} illustrates the difference in the
waiting time behavior between the sequence of Mersenne numbers,
$\{M_n\}$,  and the sequence of ``smooth Mersenne numbers'', $\{2^{n\log
n}\}$.  For the sequence of Mersenne numbers the waiting time distribution 
resembles a geometric distribution very closely, while for 
its smooth analog,  the waiting times seem to have an irregular
distribution that is quite far from a geometric distribution.


\section{Description of Data and Implementation Notes}
\label{sec:data}

\subsection{Description of data}
Our analysis is based on the leading digits of the first billion
terms of the following sequences:
\begin{itemize}
\item\textbf{Mersenne numbers.}
The Mersenne numbers, defined as $M_n=2^{p_n}-1$, where $p_n$ is the
$n$-th prime, form our main object of investigation.

\item\textbf{Random Mersenne numbers.} 
Random Mersenne numbers form one of our ``control'' sequences against
which we compare the leading digit behavior of the Mersenne numbers. 
They are defined as $M_n^*=2^{p_n^*}-1$, where
$\{p_n^*\}$ is a sequence of ``random'' primes obtained by declaring 
an integer $n\ge 3$ to be a prime with probability $1/\log n$.  

\item \textbf{Smooth Mersenne numbers.}
The sequence of ``smooth'' Mersenne numbers, defined as $2^{n\log n}$,
constitutes our second main ``control'' sequence for the Mersenne numbers.
The smooth Mersenne numbers are essentially the numbers obtained 
by replacing the $n$-th prime, $p_n$, in the definition of $M_n$ by
its smooth asymptotic, $n\log n$.

\item \textbf{Other ``smooth'' sequences.}
Additional smooth sequences we have used as points of
comparisons in some of our analyses 
are the sequences $\{2^n\}$ and $\{2^{n^2}\}$. 
\end{itemize}

\subsection{Generating the leading digits}
Because of the size of the numbers involved
(for example, the billionth Mersenne
number has more than six billion decimal digits), 
computing the necessary 
sequences of leading digits is a nontrivial task. 
For our primary sequence, the Mersenne numbers, 
we proceeded as follows:

\begin{enumerate}
\item Generate the prime numbers $p_n$, $n=1,2,\dots,10^9$,
using an optimized version of the sieve of Eratosthenes, obtained from
\url{http://primesieve.org}.

\item For each such $p_n$, compute 
$\{p_n\log_{10}2\}$, 
where $\{t\}$ denotes the fractional part of $t$,
and determine the
unique integer $d\in\{1,2,\dots,9\}$ such that 
\[
\log_{10}d\le \{p_n\log_{10}2\}<\log_{10}(d+1). 
\]
This value of $d$ is the leading digit of $2^{p_n}$ in base $10$.
To get accurate values of leading digits, we used the \url{REAL} 
data type from the C++ library iRRAM,  \url{http://irram.uni-trier.de},
a library for error-free real arithmetic (see M\"uller
\cite{muller2001irram}).

\item 
To obtain the leading digits for the Mersenne numbers $2^{p_n}-1$, 
observe that $2^{p_n}-1$ and $2^{p_n}$  have the same leading digit
unless $p_n=2$ or $p_n=3$.
Indeed, if $2^m-1$ has leading digit $d$, then 
$d\cdot 10^k\le 2^m-1<(d+1)\cdot 10^k$ for some nonnegative integer $k$,
which implies $d\cdot 10^k\le 2^m<(d+1)\cdot 10^k$  unless
$2^m=(d+1)10^k$. 
But the latter equation can only hold if $k=0$, i.e., if $2^m = d+1\le 10$.
Thus, except for the two terms corresponding to $p_1=2$ and $p_2=3$, the leading
digit of $2^{p_n}$ obtained in the previous step is also the leading digit
of  the Mersenne number $2^{p_n}-1$. Adjusting for the two exceptional terms
gives the leading digit sequence for the Mersenne numbers.
\end{enumerate}

Leading digits of the various smooth analogs of the Mersenne numbers were
generated in the same way, with the sequence of prime numbers in the
first step replaced by the appropriate smooth sequence.

\subsection{Generating random Mersenne numbers}
Random Mersenne numbers were generated in the same way as Mersenne numbers, 
with the sequence of primes, $\{p_n\}$, replaced by a sequence of random
primes,
$\{p_n^*\}$, obtained as follows:

\begin{enumerate}
\item
For each $n\ge3$ generate a random real number $R_n$ in the interval $[0,1]$.
As random number generator we used the \url{uniform_real_distribution} 
function from the C++ standard library.

\item 
If $R_n\le 1/\log n$, declare $n$ to be a random prime. 

\item
Let $p_1^*<p_2^*<\cdots<p_{10^9}^* $ be the ordered sequence of the first 
$10^9$ random primes obtained.
\end{enumerate}
The second step ensures that the random primes generated occur with
density $1/\log n$, which is the actual density 
of prime numbers near $n$. Indeed, by a classic form
of the prime number theorem (see, for example, Theorem 8.15 in
\cite{bateman-diamond2004}) we have 
\begin{equation*}
\label{eq:pnt}
\#\{n\le N: n\text{ is prime }\}=\int_2^N \frac{1}{\log x}dx +O\left(N
\exp\left(-c\sqrt{\log N}\right)\right),
\end{equation*}
where $c$ is a positive constant.   

\subsection{Implementation notes}
Most of the computations were carried out at the \url{Taub} node of the 
\emph{Illinois Campus Computing Cluster},
\url{https://campuscluster.illinois.edu/hardware/\#taub}, a multi-core high
performance computing platform running Scientific Linux 6.1, with 96 GB of
RAM.  The primary data sets of leading digits were generated and analyzed with
C++ programs, compiled with GNU \url{g++} \url{-std=c++11}. 
Generating one billion leading digits took around 5--10 hours of CPU
time.
Python 3.4 was used for additional lighter weight analysis and string
manipulation.


\section{Global Distribution Properties}

\label{sec:Mn-global}

\subsection{Empirical data on global distribution. Evidence for Conjecture
\ref{conj:benford-error}}

We begin by presenting numerical data on the frequencies of leading digits 
in the sequence of Mersenne numbers $M_n=2^{p_n}-1$.  Table
\ref{table:benford-mersenne} shows the actual counts for leading digits 
among the first billion Mersenne numbers, along with the
predicted counts based on the Benford frequencies \eqref{eq:benford}.  In
all cases, the actual and predicted counts agree to at least three digits.


\begin{table}[H]
\begin{center}
\begin{tabular}{|c|r|r|r|}
\hline
Digit& Count & Benford Prediction & Error
\\
\hline
\hline
1 & 301032256 & 301029995.66 & 2260.34  
\\
\hline
2 & 176095018 & 176091259.06 & 3758.94  
\\
\hline
3 & 124946964 & 124938736.61 & 8227.39  
\\
\hline
4 & 96901940 & 96910013.01 & -8073.01 
\\
\hline
5 & 79176717 & 79181246.05 & -4529.05  
\\
\hline
6 & 66950369 & 66946789.63 & 3579.37  
\\
\hline
7 & 57993513 & 57991946.98 & 1566.02  
\\
\hline
8 & 51145193 & 51152522.45 & -7329.45  
\\
\hline
9 & 45758030 & 45757490.56 & 539.44  
\\
\hline
\end{tabular}
\end{center}
\caption{%
Actual versus predicted counts of leading digits among the first $10^9$
Mersenne numbers. The predicted counts are given by 
$NP(d)$, where $N=10^9$ is the number of terms, $d$ is the
digit, and $P(d)=\log_{10}(1+1/d)$ is the Benford frequency for digit $d$,
given by \eqref{eq:benford}.
}
\label{table:benford-mersenne}
\end{table}


The errors in Table \ref{table:benford-mersenne} appear to be  
roughly of the squareroot size predicted by Conjecture
\ref{conj:benford-error}.  For a more detailed analysis,
we compare these errors to those in a random model in which
the events ``$M_n$ has leading digit $d$'', $n=1,2,\dots$,
are assumed to be independent events with probability $P(d)=\log_{10}(1+1/d)$.
Under this assumption, the Central Limit Theorem yields that
the number of terms with leading digit $d$ among
the first $N$ terms $M_n$ is approximately normally distributed with mean
$NP(d)$ and standard deviation $\sqrt{N\cdot P(d)(1-P(d))}$.  This motivates
normalizing the error to a z-score, defined as 
\begin{equation}
\label{eq:z-score}
z(d,N)=\frac{\#\{n\le N: D(M_n)=d\} - N P(d)}{\sqrt{N\cdot P(d)(1-P(d))}}.
\end{equation}
Figure \ref{fig:benford-global-z-score} shows the behavior of 
these z-scores as functions
of $N$ for the digits $1$, $2$, $5$, and $9$. In all cases, the z-scores
exhibit the typical random walk type behavior associated with sequences of
independent Bernoulli random variables.

\begin{figure}[H]
\begin{center}
\includegraphics[width=.45\textwidth]{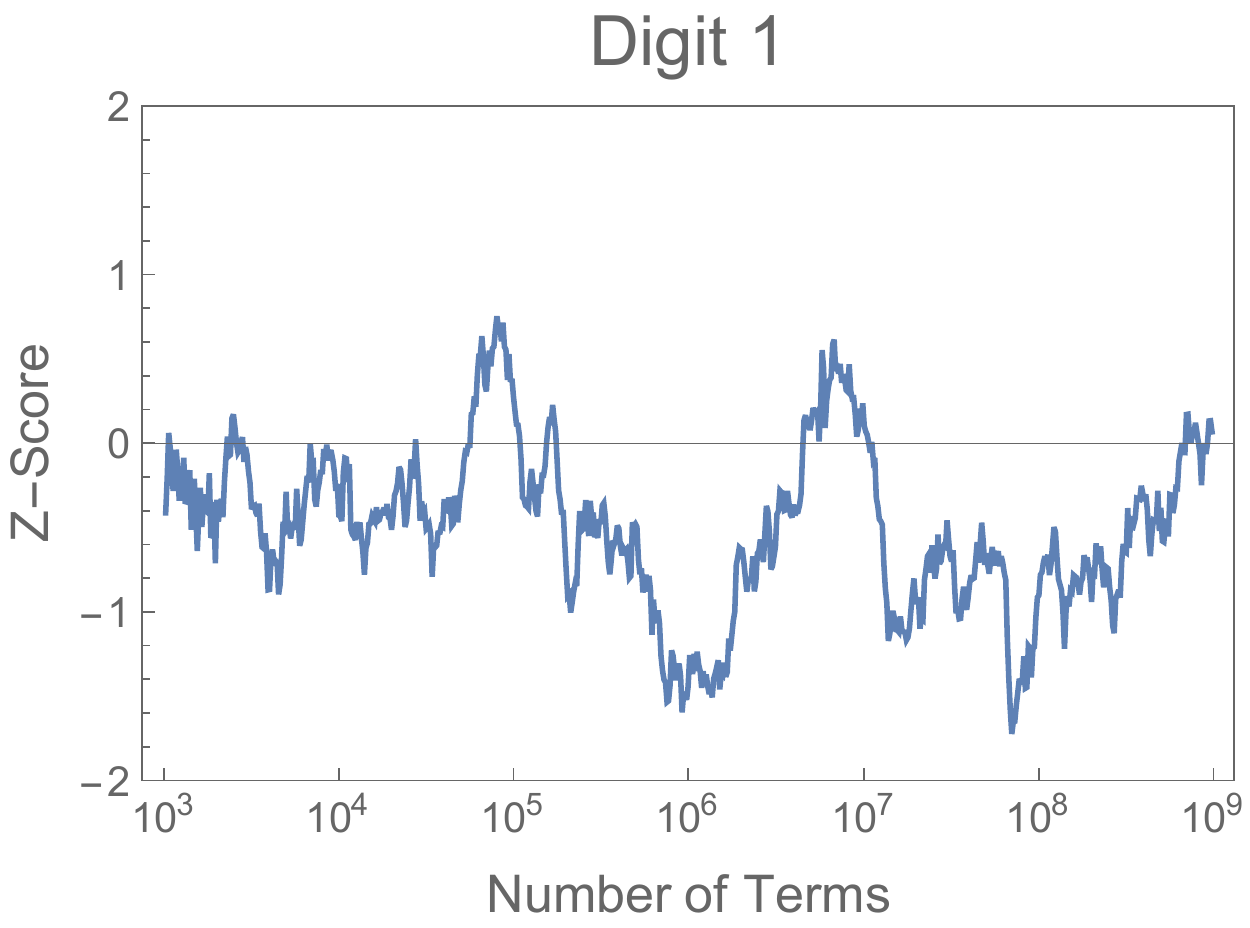}
\hspace{1em}
\includegraphics[width=.45\textwidth]{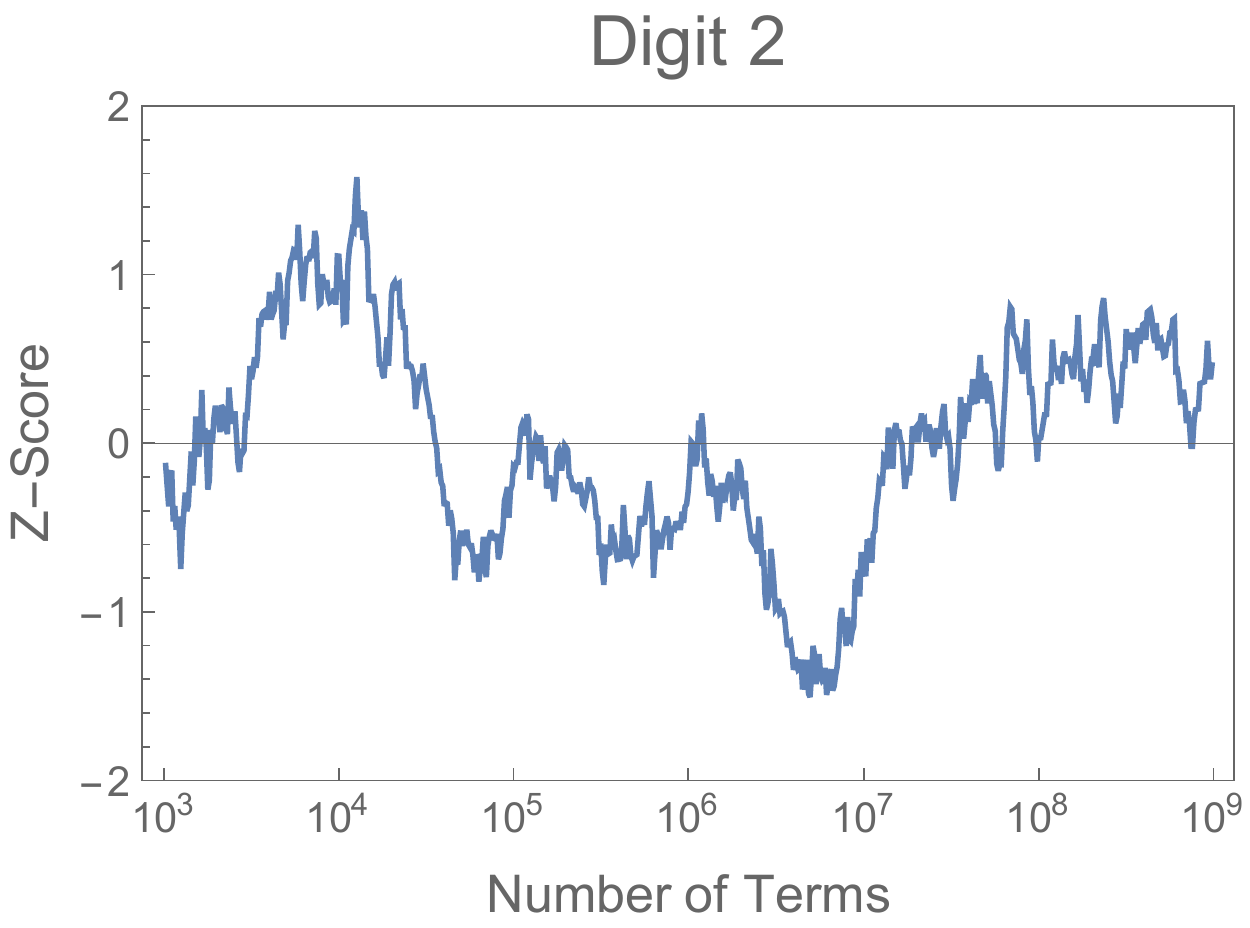}
\\[1ex]
\includegraphics[width=.45\textwidth]{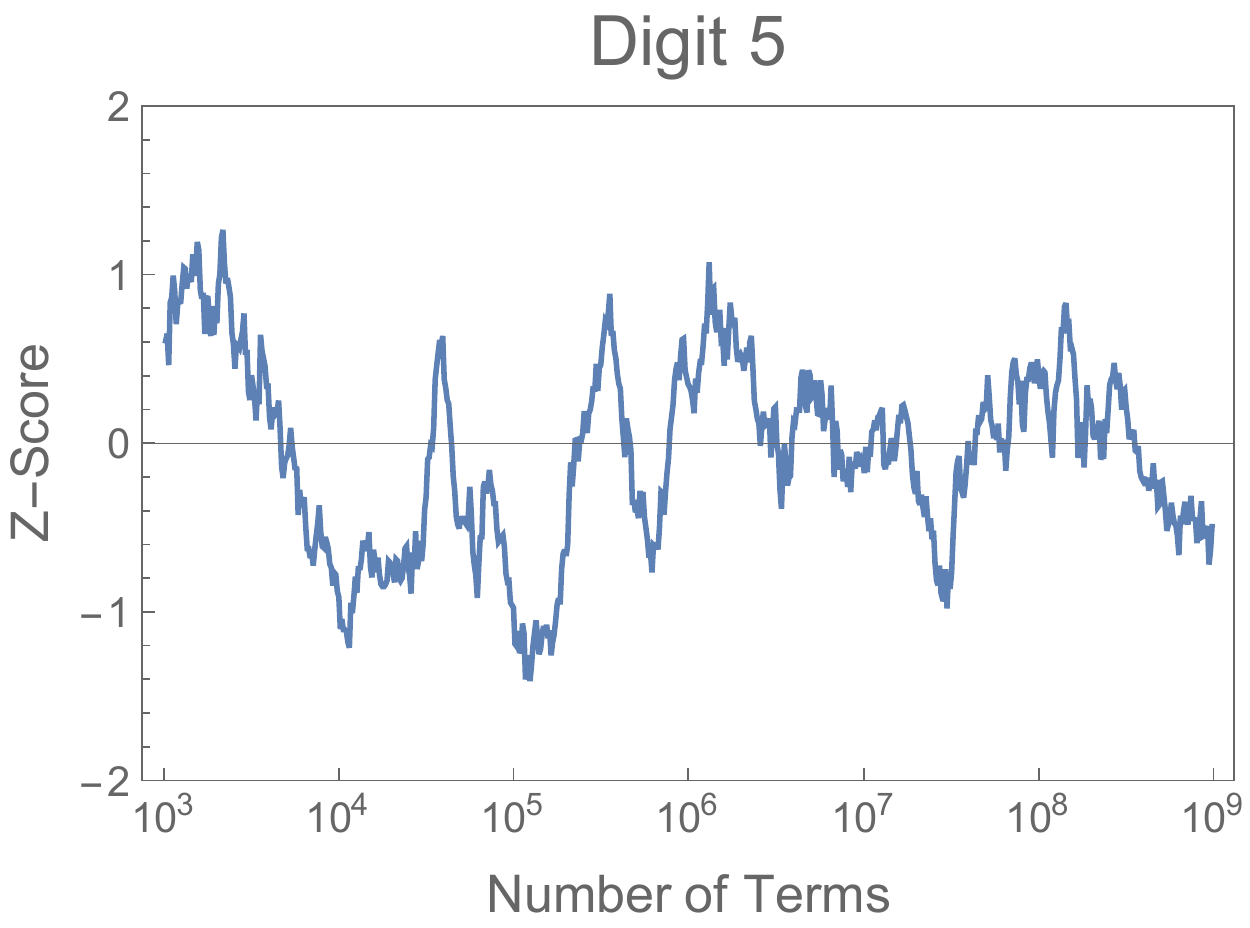}
\hspace{1em}
\includegraphics[width=.45\textwidth]{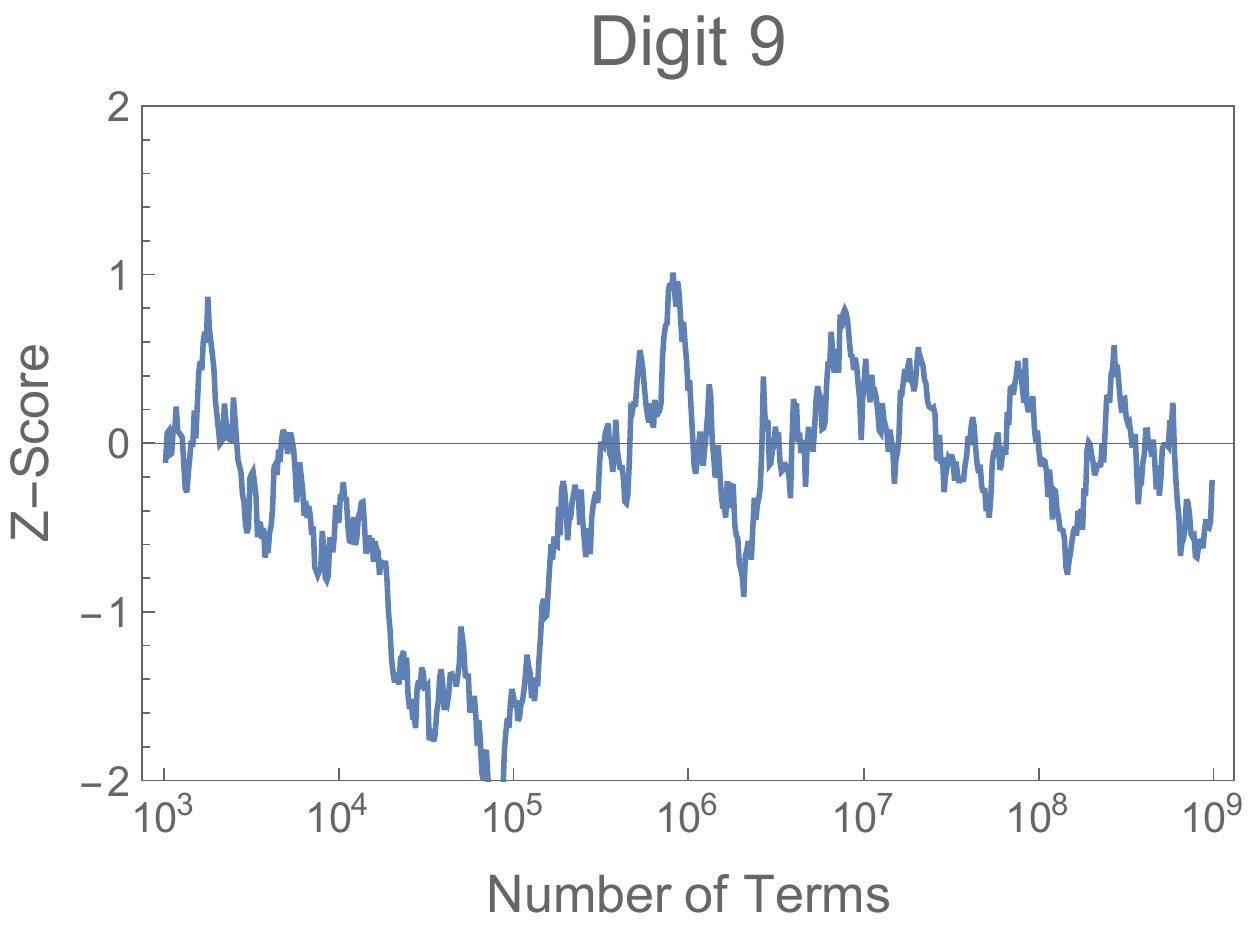}
\end{center}
\caption{%
The behavior of the normalized errors, given by the z-scores 
\eqref{eq:z-score}, in the leading digit counts for Mersenne numbers
for digits $1$, $2$, $5$, and $9$.
In all cases, the z-scores exhibit the typical random-like behavior one would
observe if the digits were to occur independently with Benford frequencies
\eqref{eq:benford}.
}
\label{fig:benford-global-z-score}
\end{figure}

Additional insight is provided by the two charts 
of Figure \ref{fig:benford-global-histogram}, which show the distribution 
of z-scores $z(d,N)$, as $N$ runs through the  
geometrically spaced sample points $N_i=
\lfloor 10^3\cdot 1.05^i\rfloor$, $i=1,2,\dots,696$.
The box-whisker chart on the left shows, for each
digit $d$, the median, quartiles, and extreme values of the corresponding
set of z-scores.
The histogram on the right 
displays the combined distribution of these z-scores for all
digits $d$.   In particular, the data shows that all z-scores sampled fall 
into the interval $[-3,3]$ and that most are spread out over the
subinterval $[-1,1]$. In other words, within the data we have sampled
the error in the Benford approximation for leading digit counts 
is within a factor $\pm3$ of $\sqrt{N\cdot P(d)(1-P(d))}$, and  the
latter quantity represents the ``typical'' size of the error. 
This lends strong support to Conjecture
\ref{conj:benford-error}, which states that the error is
order $O(N^{1/2+o(1)})$, but not of smaller order of magnitude.

In fact, it may be the case that these errors, after normalizing by
$\sqrt{N\cdot P(d)(1-P(d))}$, converge, in an appropriate sense (for
example, in the sense of logarithmic density), to a standard normal
distribution.  This would represent a significant refinement of
Conjecture \ref{conj:benford-error}.

\begin{figure}[H]
\begin{center}
\includegraphics[width=.45\textwidth]{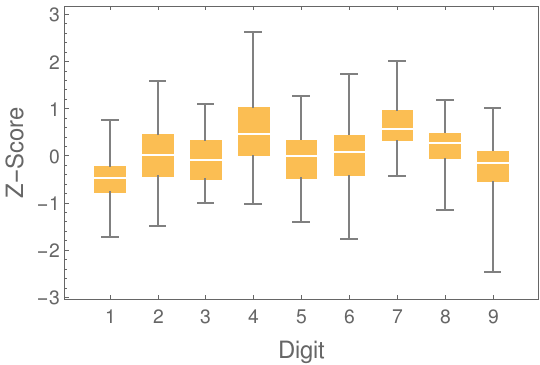}
\hspace{1em}
\includegraphics[width=.45\textwidth]{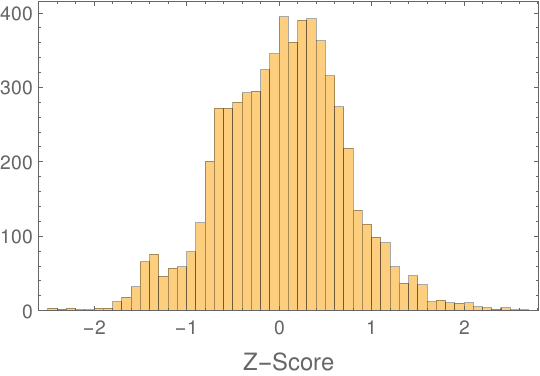}
\end{center}
\caption{%
The distribution of the normalized errors, given by the z-scores 
\eqref{eq:z-score}, in the leading digit counts for Mersenne numbers.
The box-whisker chart shows the distribution of z-scores $z(d,N)$ for each
individual digit $d$. 
The histogram shows the combined distribution of these z-scores 
over all digits $d$.
}
\label{fig:benford-global-histogram}
\end{figure}

\subsection{Proof of Theorem \ref{thm:Mn-global}}

To conclude this section, we 
show that the sequence $\{M_n\}$ satisfies Benford's Law.
The proof is short, but it relies on two key results from the literature.
The first result introduces an important tool in establishing Benford's Law for
mathematical sequences, namely the concept of uniform distribution modulo $1$.

\begin{defn}[Uniform Distribution Modulo $1$]
\label{def:ud-mod1}
A sequence of real numbers
$\{u_n\}$ is called \emph{uniformly distributed modulo $1$} if it
satisfies
\begin{equation}
\label{eq:def-ud-mod1}
\lim_{N\to\infty} \frac1N \#\{n\le N: \{u_n\}\le t
\}=t\quad (0\le t\le 1),
\end{equation}
where $\{x\}$ denotes the fractional part of $x$.
\end{defn}

The connection between uniform distribution and 
Benford's Law is given in the following proposition,  
due to Diaconis \cite[Theorem 1]{diaconis1977}.

\begin{prop}[Diaconis]
\label{prop:benford-ud-mod1}
Let $\{a_n\}$ be a sequence of positive real numbers. If the sequence
$\{\log_{10} a_n\}$ is uniformly distributed modulo $1$, then $\{a_n\}$
satisfies Benford's Law.
\end{prop}

The second ingredient in the proof is the following  result of Vinogradov
\cite{vinogradov1947} (see also Iwaniec and Kowalski \cite[Theorem
21.3]{iwaniec2004}).

\begin{prop}[Vinogradov]
\label{prop:primes-ud-mod1}
Let $\alpha$ be an irrational number. Then the sequence $\{\alpha p_n\}$
is uniformly distributed modulo $1$. 
\end{prop}

\begin{proof}[Proof of Theorem \ref{thm:Mn-global}]
Let $u_n=\log_{10}M_n$. 
By Proposition \ref{prop:benford-ud-mod1}, 
to show that $\{M_n\}$ satisfies Benford's Law, 
it suffices to show that the sequence $\{\log_{10}M_n\}$ is uniformly
distributed modulo $1$. Now,
\begin{equation}
\label{eq:Mn-pn}
\log_{10}M_n=\log_{10}(2^{p_n}-1)=(\log_{10}2)p_n+o(1)\quad (n\to\infty).
\end{equation}
Since $\log_{10}2$ is irrational, Proposition \ref{prop:primes-ud-mod1}
implies that the sequence $\{(\log_{10}2)p_n\}$ is uniformly distributed 
modulo $1$. To conclude that $\{\log_{10}M_n\}$ is also uniformly
distributed modulo $1$, it suffices to observe that 
if $\{u_n\}$ is uniformly distributed modulo $1$, then any sequence
$\{u_n^*\}$ satisfying $u_n^*=u_n+o(1)$ as $n\to\infty$ is also uniformly
distributed modulo $1$. The latter claim follows immediately from the
definition \eqref{eq:def-ud-mod1} of uniform distribution modulo $1$.
\end{proof}

\begin{remark}
\label{rem:vinogradov}
An interesting question is to what extent the error in the uniform
distribution result of Proposition \ref{prop:primes-ud-mod1}
depends on diophantine approximation properties of
$\alpha$. We are not aware of any results of this 
type in the literature, though such error bounds could in principle
be obtained from appropriate exponential sum estimates (e.g.,
Theorem 4.2 in Banks and Shparlinski \cite{banks2009}) via the
Erd\"os-Turan inequality as in the  proof of \cite[Chapter 2, Theorem
3.2]{kuipers1974}. The resulting error bounds would, however, likely be
quite far from being best-possible. 
\end{remark}


\section{Local Distribution Properties}
\label{sec:Mn-local}

\subsection{Empirical data on waiting times. 
Evidence for Conjecture \ref{conj:benford-waits}}

We begin by providing data in support of 
the waiting time conjecture, Conjecture \ref{conj:benford-waits}.
Figure \ref{fig:benford-waits-details} shows the distribution of
waiting times between occurrences of leading digit $1$ 
among the first $10^9$ terms of the sequence of Mersenne numbers, 
along with the analogous distributions for 
three ``control'' sequences: ``random'' Mersenne numbers,
$2^{p_n^*}-1$, where $\{p_n^*\}$ is a sequence of random primes (see
Section \ref{sec:data});  ``smooth'' Mersenne numbers,
defined as $2^{n\log n}$; and the sequence $\{2^{n^2}\}$.

For each of these sequences,
the observed frequencies for waiting times $1,2,\dots,10$
are shown along with the ``theoretical'' frequencies, 
given by 
\begin{equation}
\label{eq:waiting-time-prediction}
P(\text{waiting time between $1$'s is $k$})= p(1-p)^{k-1}, \quad
k=1,2,\dots, 
\end{equation}
where $p=\log_{10}2$ is the Benford probability for leading digit $1$.  
For the Mersenne numbers and random Mersenne
numbers the agreement between the predicted and actual waiting time
distribution is very good.  By contrast, the smooth analogs of $M_n$,
shown in the bottom two charts of Figure \ref{fig:benford-waits-details},
have a noticeably different waiting time distribution.


\begin{figure}[H]
\begin{center}
\includegraphics[width=.45\textwidth]{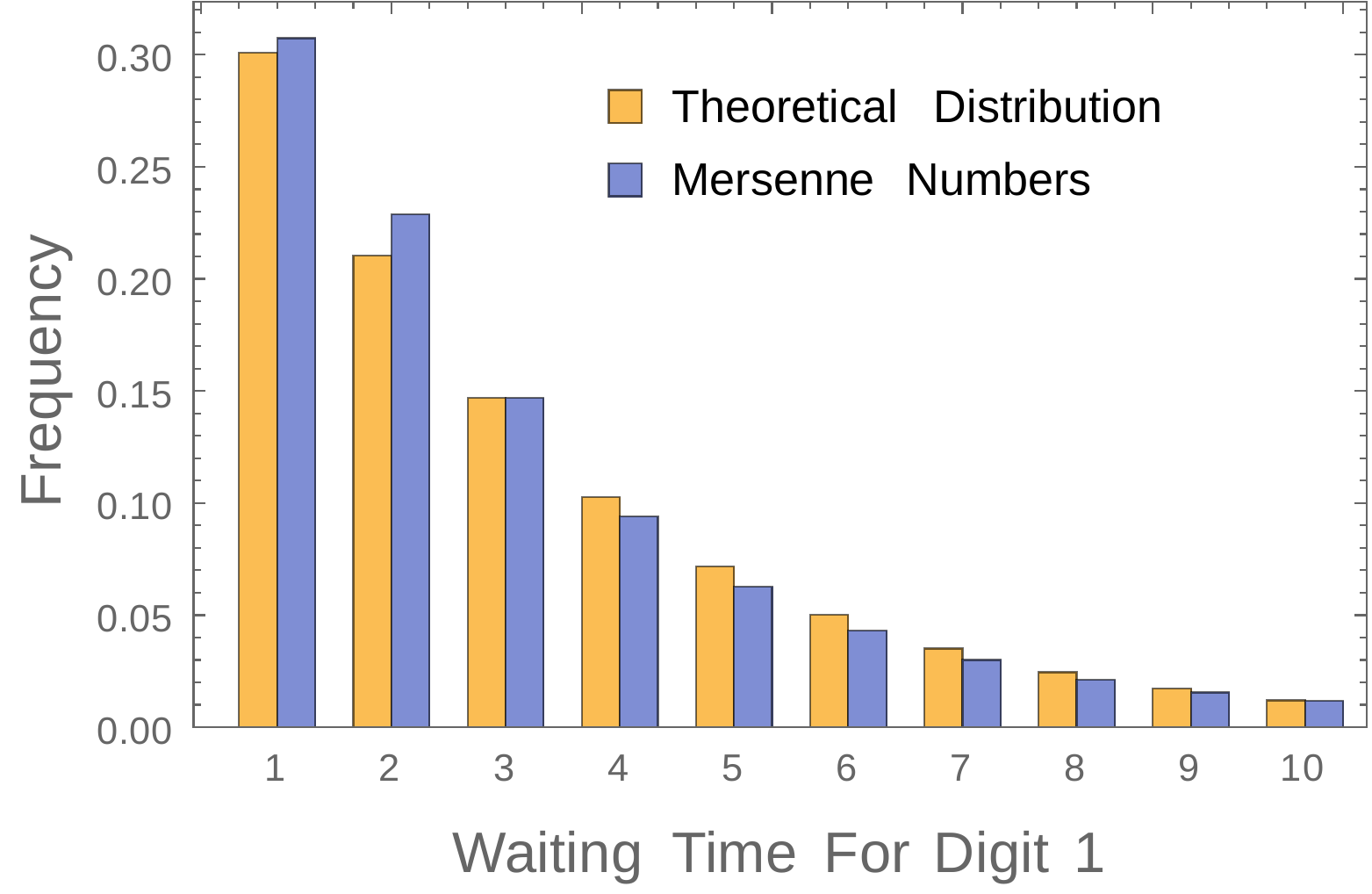}
\hspace{1em}
\includegraphics[width=.45\textwidth]{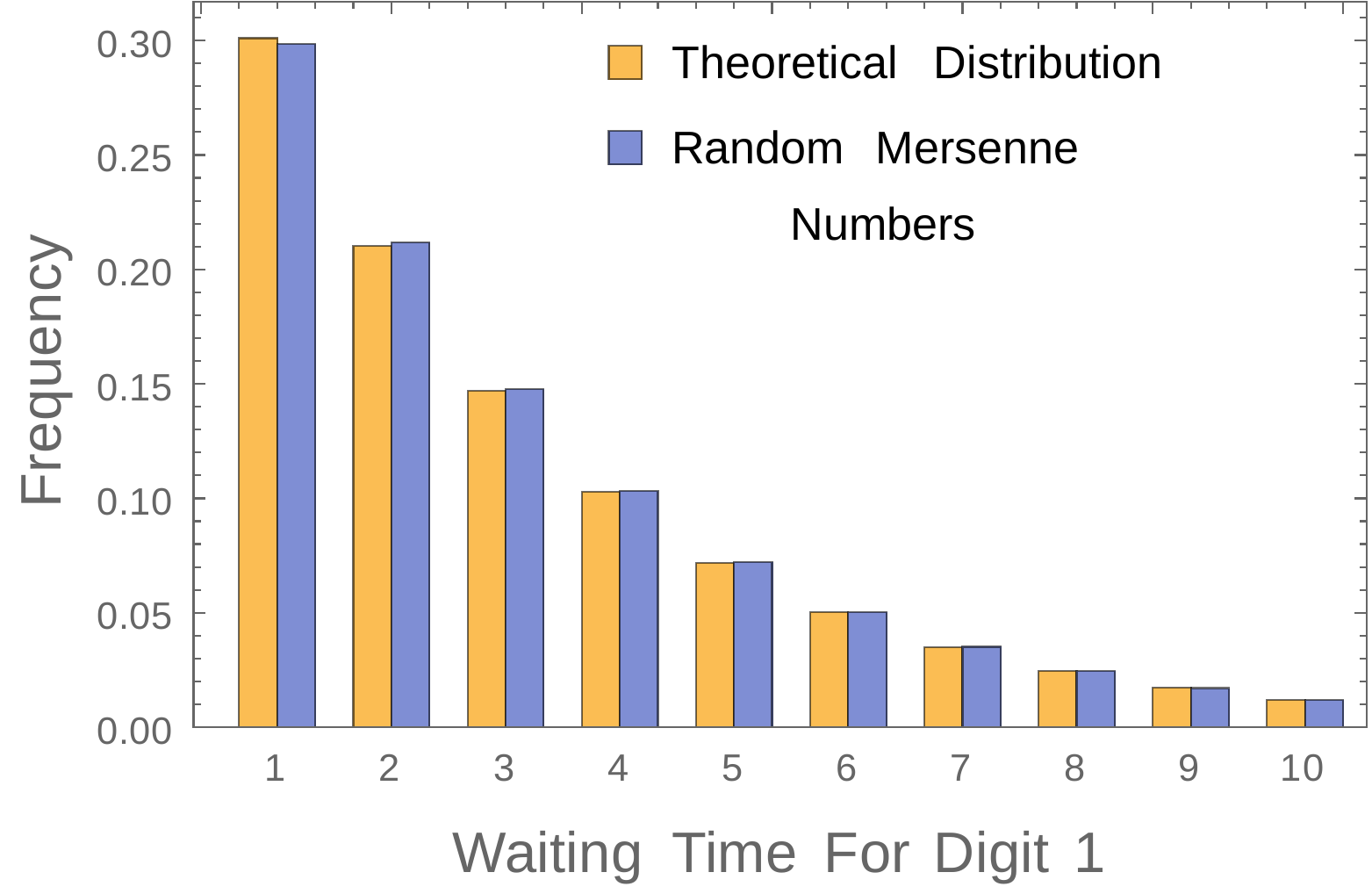}
\\[1ex]
\includegraphics[width=.45\textwidth]{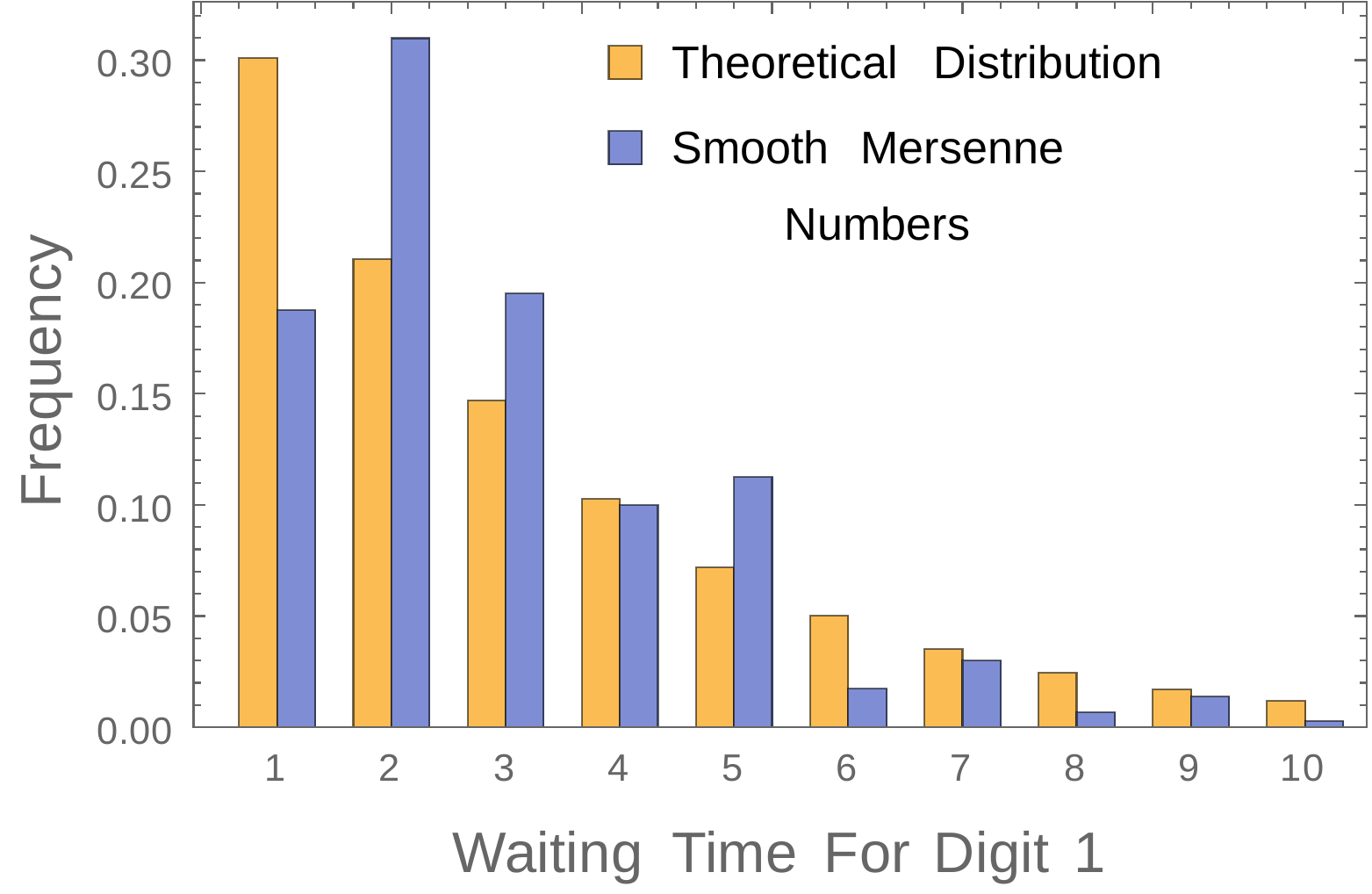}
\hspace{1em}
\includegraphics[width=.45\textwidth]{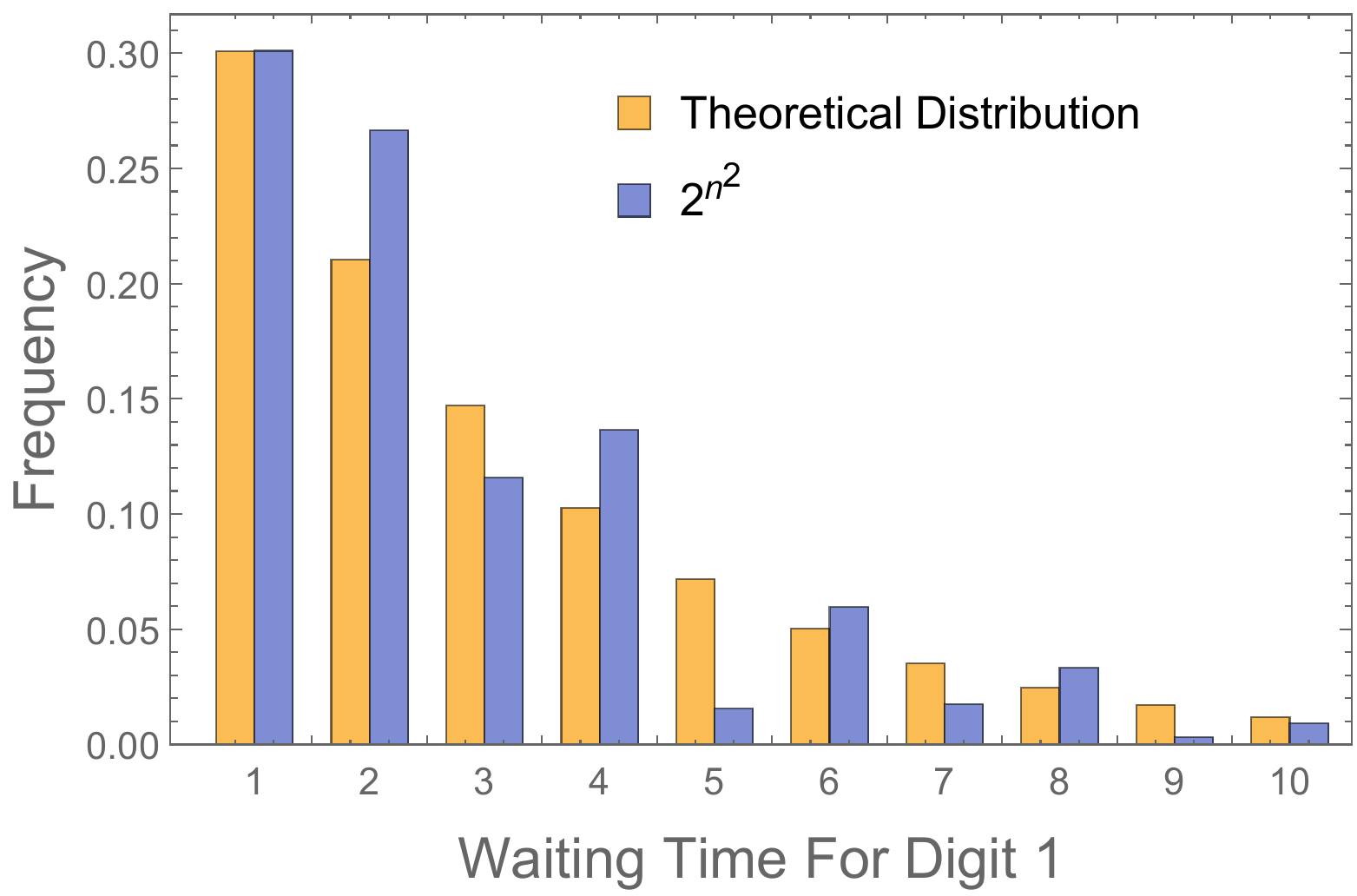}
\end{center}
\caption{The distribution of ``waiting times'' between occurrences of 
leading digit $1$ among the first $10^9$ terms 
of four sequences: the 
Mersenne numbers $M_n=2^{p_n}-1$; ``random'' Mersenne numbers,
$2^{p_n^*}-1$, where $p_n^*$ denote ``random'' primes;  
``smooth'' Mersenne numbers, $2^{n\log n}$; and the 
sequence $\{2^{n^2}\}$. Of these four sequences only the
Mersenne sequence and its random analog 
exhibit a geometric waiting time distribution.
}
\label{fig:benford-waits-details}
\end{figure}

For another perspective on the behavior of the waiting times between
leading digits,  we consider the waiting time frequencies 
as functions of the number of terms in the sequence. 
For the Mersenne sequence the results are shown in  
Figure \ref{fig:mersenne-waits-k-all}:  The observed frequencies are
close to the predicted frequencies at all sample points, and the
agreement improves as the number of terms increases.

\begin{figure}[H]
\begin{center}
\includegraphics[width=.45\textwidth]{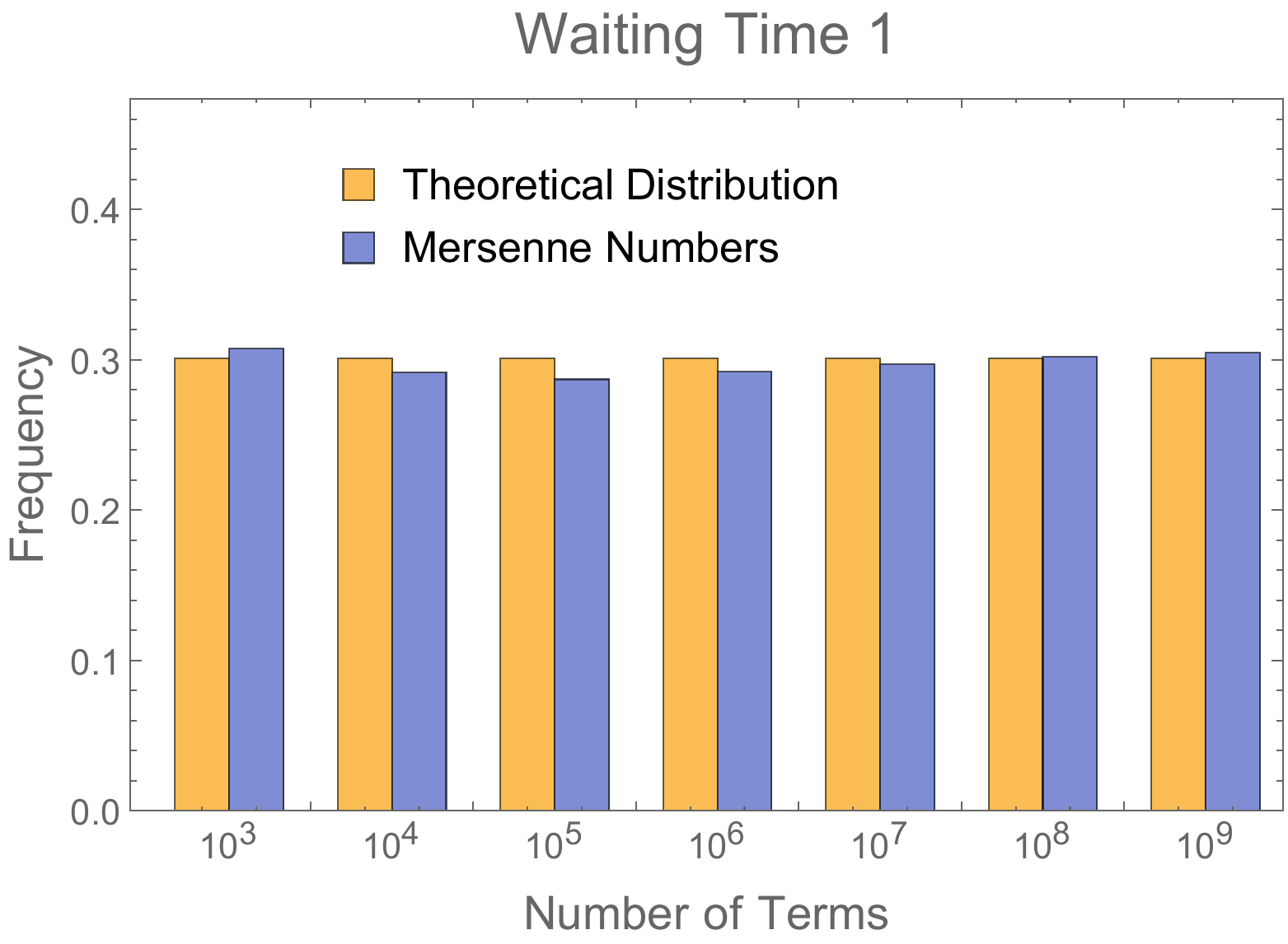}
\hspace{1em}
\includegraphics[width=.45\textwidth]{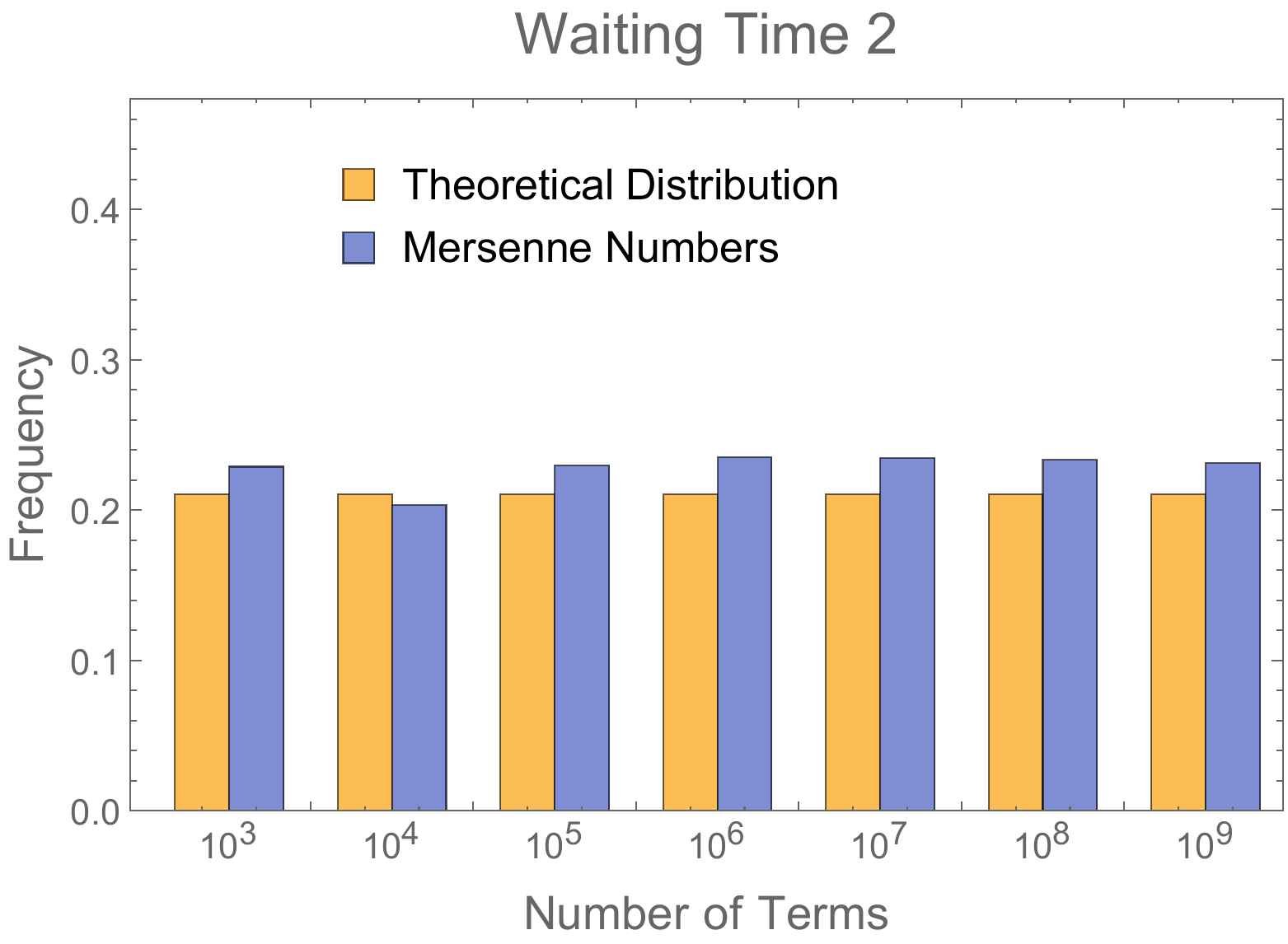}
\\[1ex]
\includegraphics[width=.45\textwidth]{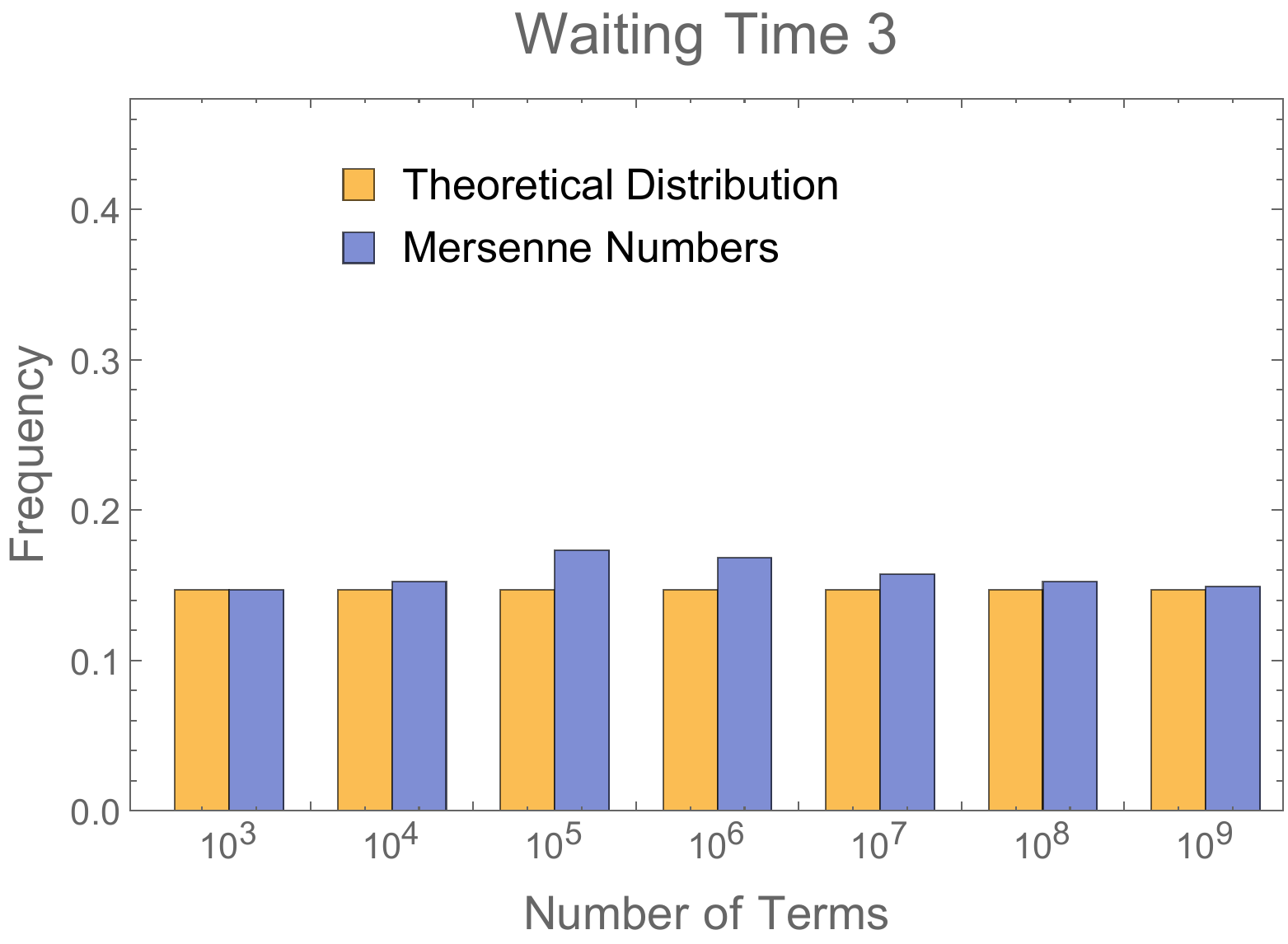}
\hspace{1em}
\includegraphics[width=.45\textwidth]{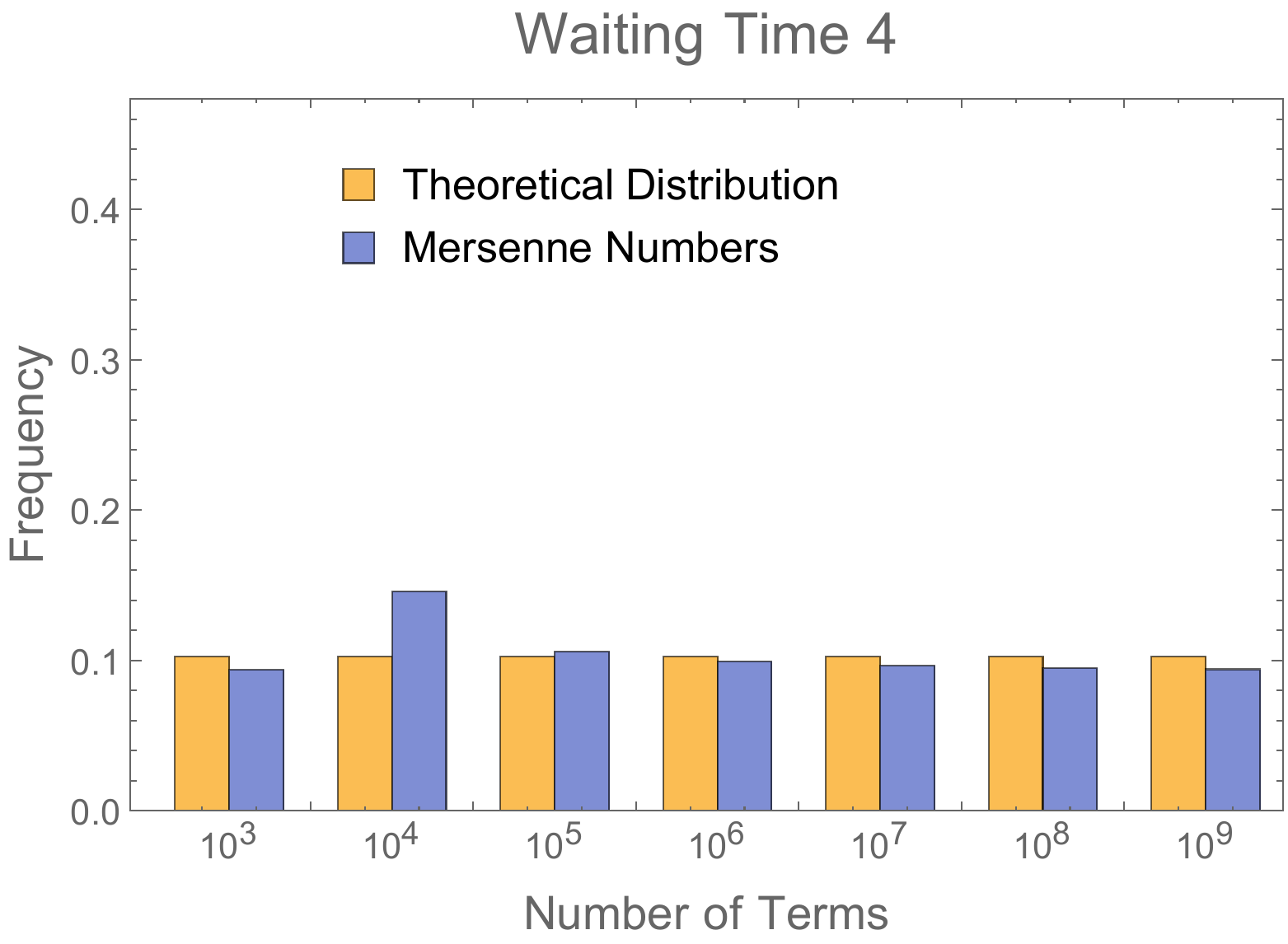}
\end{center}
\caption{Frequencies of waiting time $k$ (where $k=1,2,3,4$)
between occurrences of 
leading digit $1$ among the first $10^i$ Mersenne numbers 
($i=3,\dots,9$), along with the theoretical frequencies given by
\eqref{eq:waiting-time-prediction}.
}
\label{fig:mersenne-waits-k-all}
\end{figure}

Figure \ref{fig:smooth-mersenne-waits-k-all} shows the same set of waiting
time frequencies for the sequence of smooth Mersenne numbers.  
Here the behavior is completely different from the case of Mersenne
numbers. Not only do the observed frequencies differ significantly from
the theoretical distribution, they also exhibit large oscillations
as the number of terms increases, suggesting that a limit distribution
does not exist.

\begin{figure}[H]
\begin{center}
\includegraphics[width=.45\textwidth]{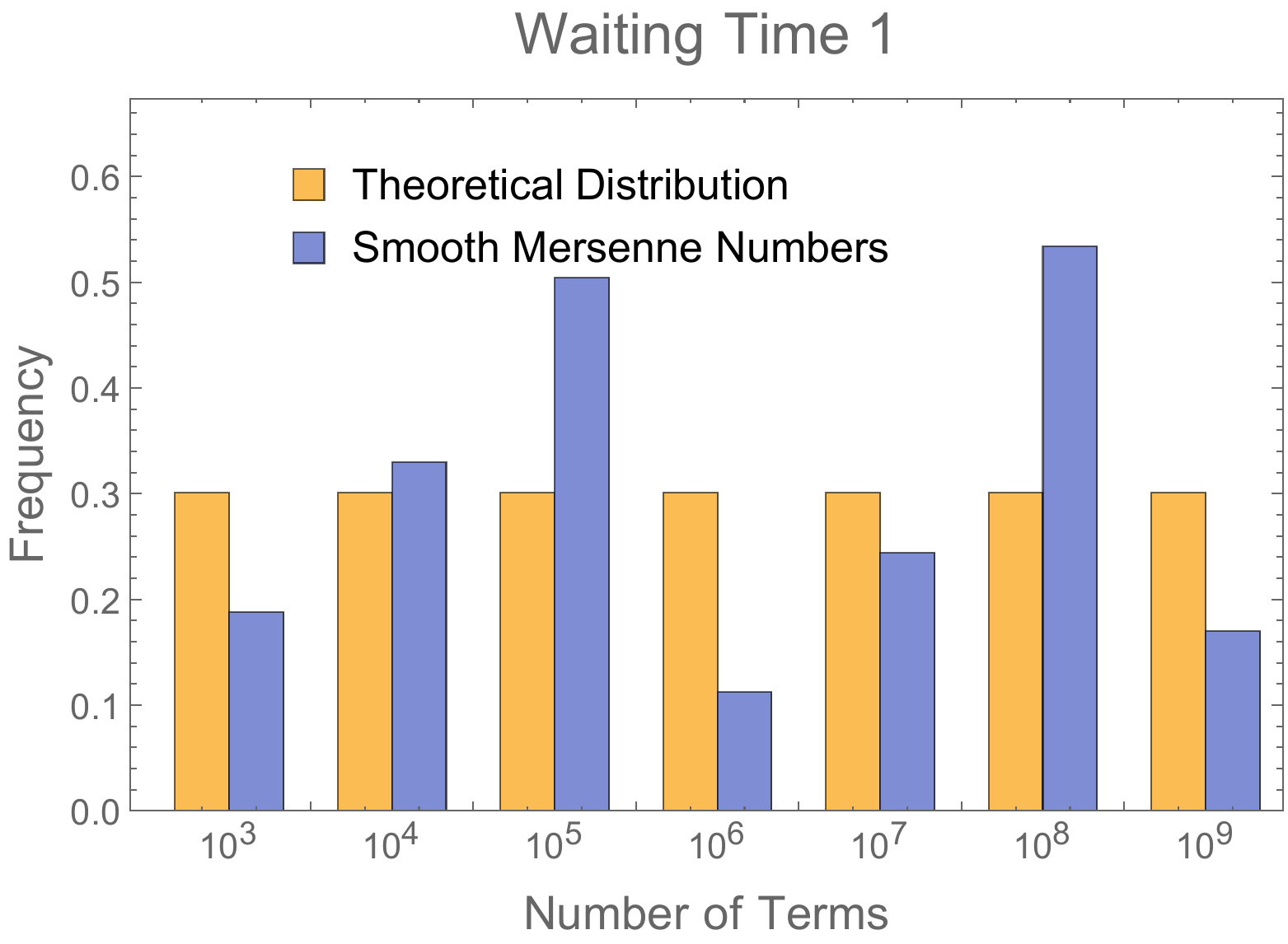}
\hspace{1em}
\includegraphics[width=.45\textwidth]{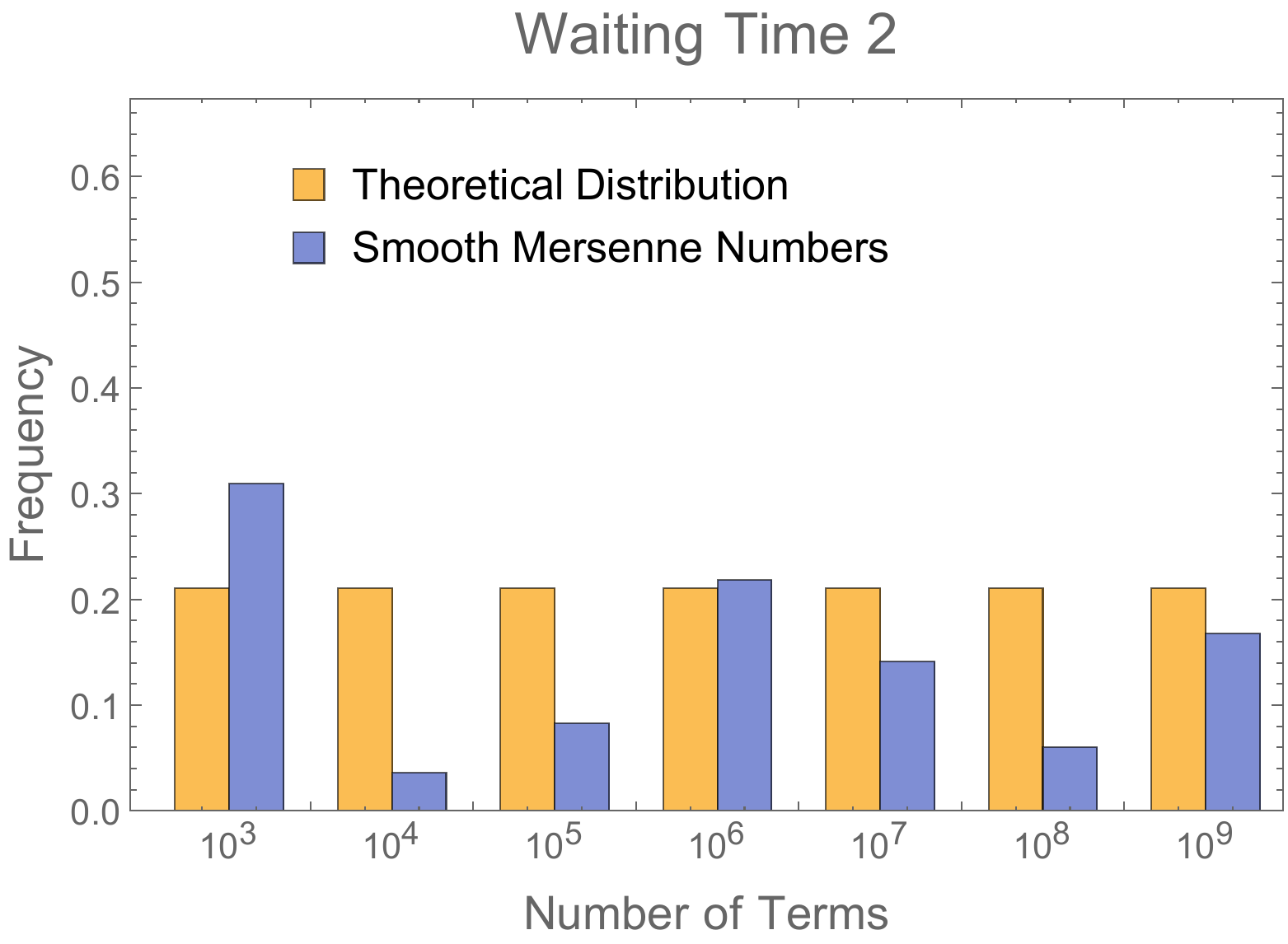}
\\[1ex]
\includegraphics[width=.45\textwidth]{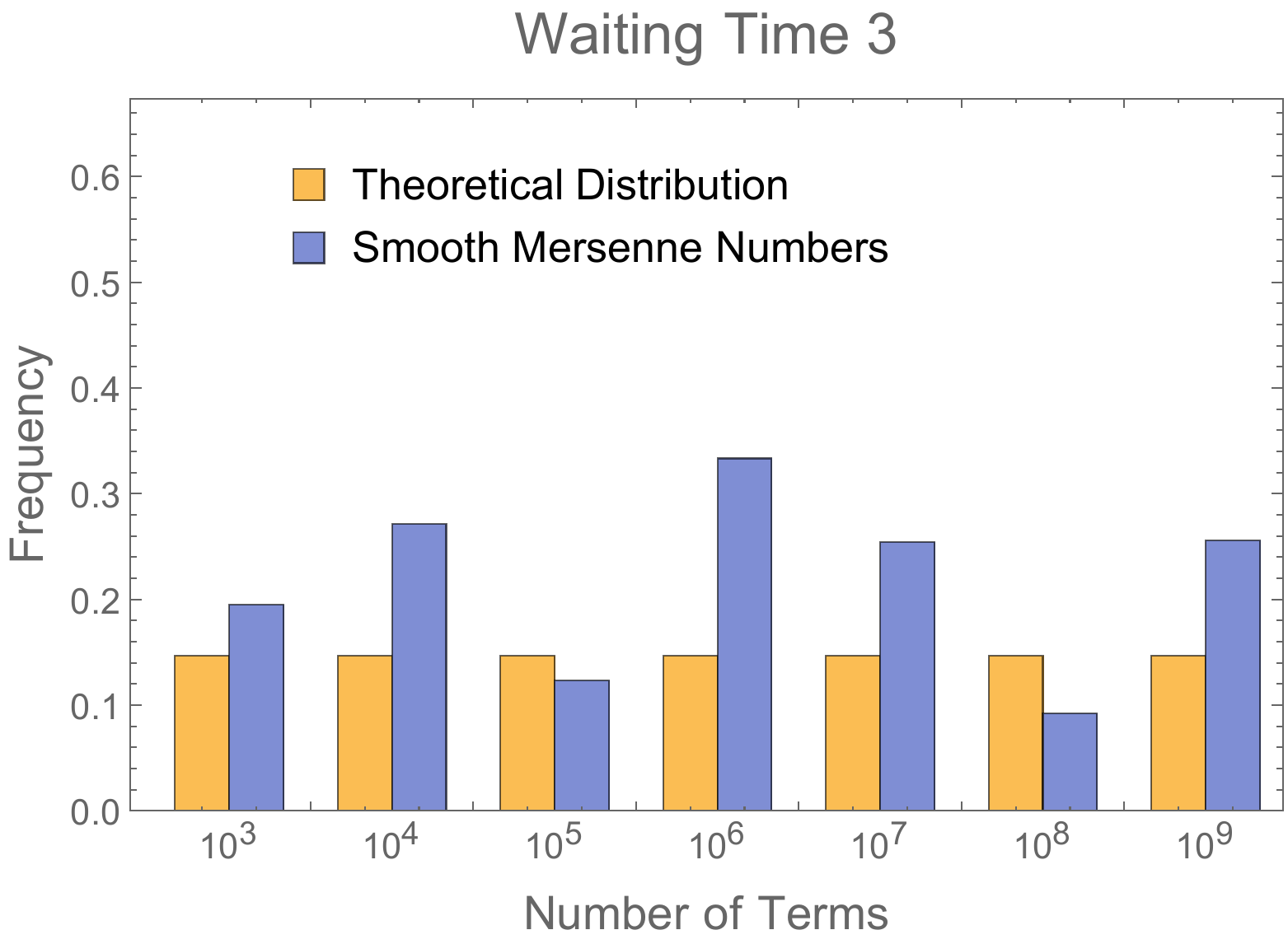}
\hspace{1em}
\includegraphics[width=.45\textwidth]{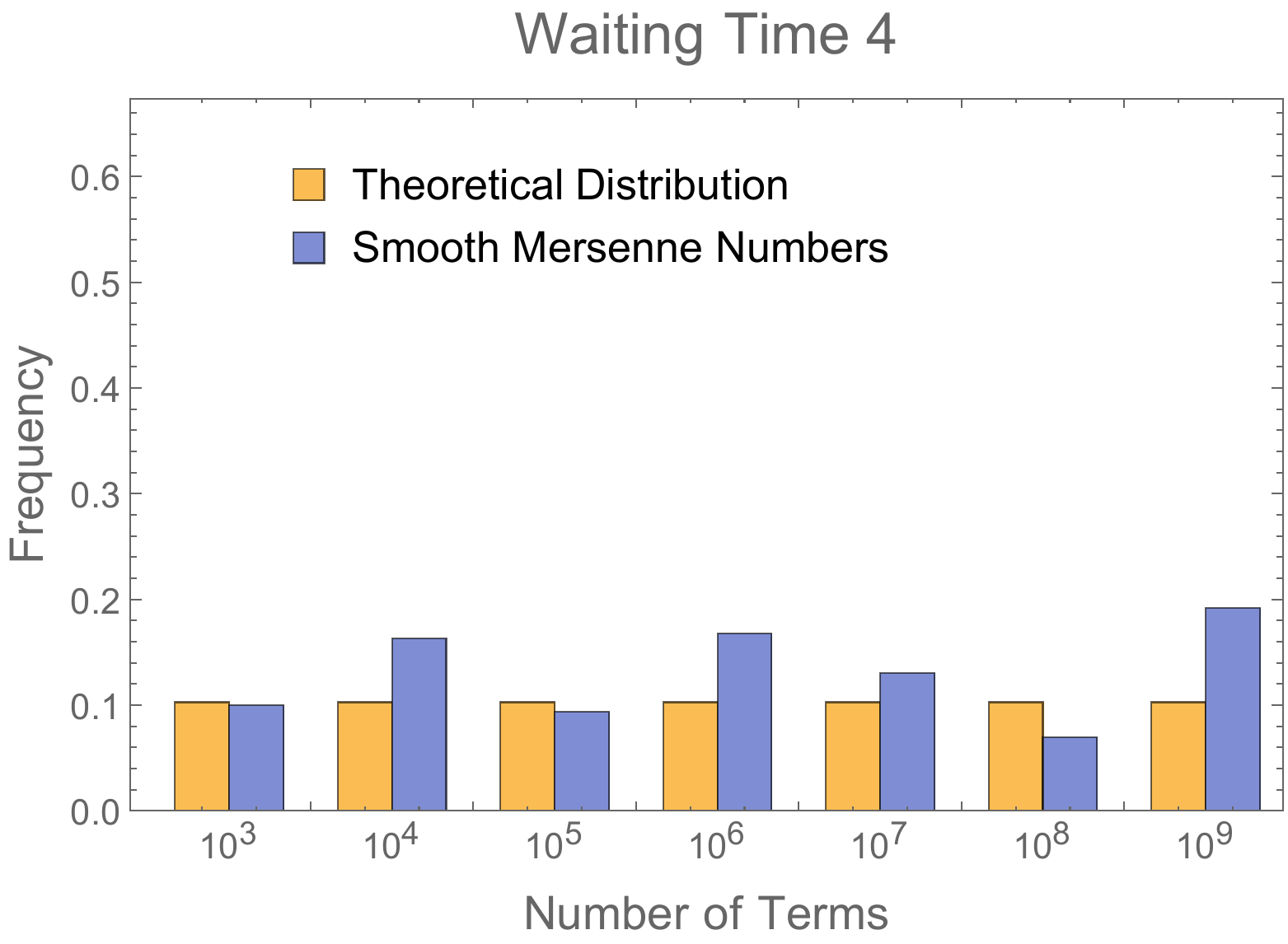}
\end{center}
\caption{Frequencies of waiting time $k$  (where $k=1,2,3,4$) between occurrences of 
leading digit $1$ among the first $10^i$ \emph{smooth} Mersenne numbers 
($i=3,\dots,9$), along with the theoretical frequencies given by
\eqref{eq:waiting-time-prediction}.
}
\label{fig:smooth-mersenne-waits-k-all}
\end{figure}

\subsection{Empirical data on local distribution of leading digits.
Evidence for Conjecture \ref{conj:local-benford}}

We next consider Conjecture \ref{conj:local-benford}, which predicts
that, for any fixed positive integer $k$, $k$-tuples of leading digits of
consecutive terms behave like $k$ independent Benford-distribution random
variables.  That is, 
each tuple $(d_1,\dots,d_k)$ of leading digits is predicted to occur with
asymptotic frequency
\begin{equation}
\label{eq:k-tuples-expected}
P(d_1,\dots,d_k)=\prod_{i=1}^kP(d_i)=\prod_{i=1}^k
\log_{10}\left(1+\frac1{d_i}\right).
\end{equation}

For $k=1$, \eqref{eq:k-tuples-expected} reduces to the (global) Benford
distribution, which we had considered in Section \ref{sec:Mn-global}. The
case $k=2$ therefore is the first test case for local Benford
distribution properties of a sequence. We have focused our numerical
computations on this case as for larger $k$-values the densities become
too small to yield meaningful numerical data within the computable range.
However, indirect evidence that the behavior predicted by Conjecture 
\ref{conj:local-benford} also holds when  $k\ge3$ 
is provided by our data on waiting time distributions:  The frequencies
for waiting time $k$ between occurrences of a given leading digit 
depend on the joint distribution of $(k+1)$-tuples of leading digits. Thus,
the close agreement that we found between observed and predicted waiting
time frequencies for the sequence of Mersenne numbers suggests that this
sequence does indeed have the predicted joint distribution
\eqref{eq:k-tuples-expected}.

Table \ref{table:benford-pairs-combined}
shows the observed and predicted 
frequencies of four pairs $(d_1,d_2)$ of leading digits
among the first $10^9$ terms of the sequences of Mersenne
numbers, random Mersenne numbers, and smooth Mersenne numbers.


\begin{table}[H]
\begin{center}
\begin{tabular}{|c|c|c|c|c|}
\hline
$(d_1,d_2)$ & 
Prediction & 
Mersenne&
Random Mersenne&
Smooth Mersenne
\\
\hline
\hline

$(1,1)$ &
0.09062 &  
0.09252 &
0.08989 &
0.05748 
\\
\hline

$(1,2)$ &
0.05301 &
0.04916 &
0.05072 &
0.07228 
\\
\hline

$(2,1)$ &
0.05300 &
0.05811 &
0.05447 &
0.01503 
\\
\hline

$(2,2)$ &
0.03101 &
0.02905 &
0.02963 &
0.01863
\\
\hline
\end{tabular}
\end{center}
\caption{
Actual versus predicted frequencies of
selected pairs $(d_1,d_2)$ of leading digits 
among the first $10^9$ terms of the sequences of Mersenne numbers,
random Mersenne numbers, 
and smooth Mersenne numbers.
The predicted frequencies are those given 
by \eqref{eq:k-tuples-expected}.
}
\label{table:benford-pairs-combined}
\end{table}


Figure \ref{fig:benford-pairs-combined} shows these frequencies as a
function of the number of terms in the sequence.
For the Mersenne and random Mersenne numbers, the
frequencies indeed seem to converge to their predicted values,
\eqref{eq:k-tuples-expected}, thus lending support to Conjecture
\ref{conj:local-benford}.
On the other hand, for the smooth Mersenne numbers, the behavior is
completely different: The 
frequencies of pairs of leading digits exhibit a distinct oscillating
behavior and do not seem to converge to a limit. Moreover, they
do not seem to be symmetric; for example, the pair $(1,2)$ occurs about
four times as often as the pair $(2,1)$ among the first $10^9$ terms in
the sequence.


\begin{figure}[H]
\begin{center}
\includegraphics[width=.45\textwidth]{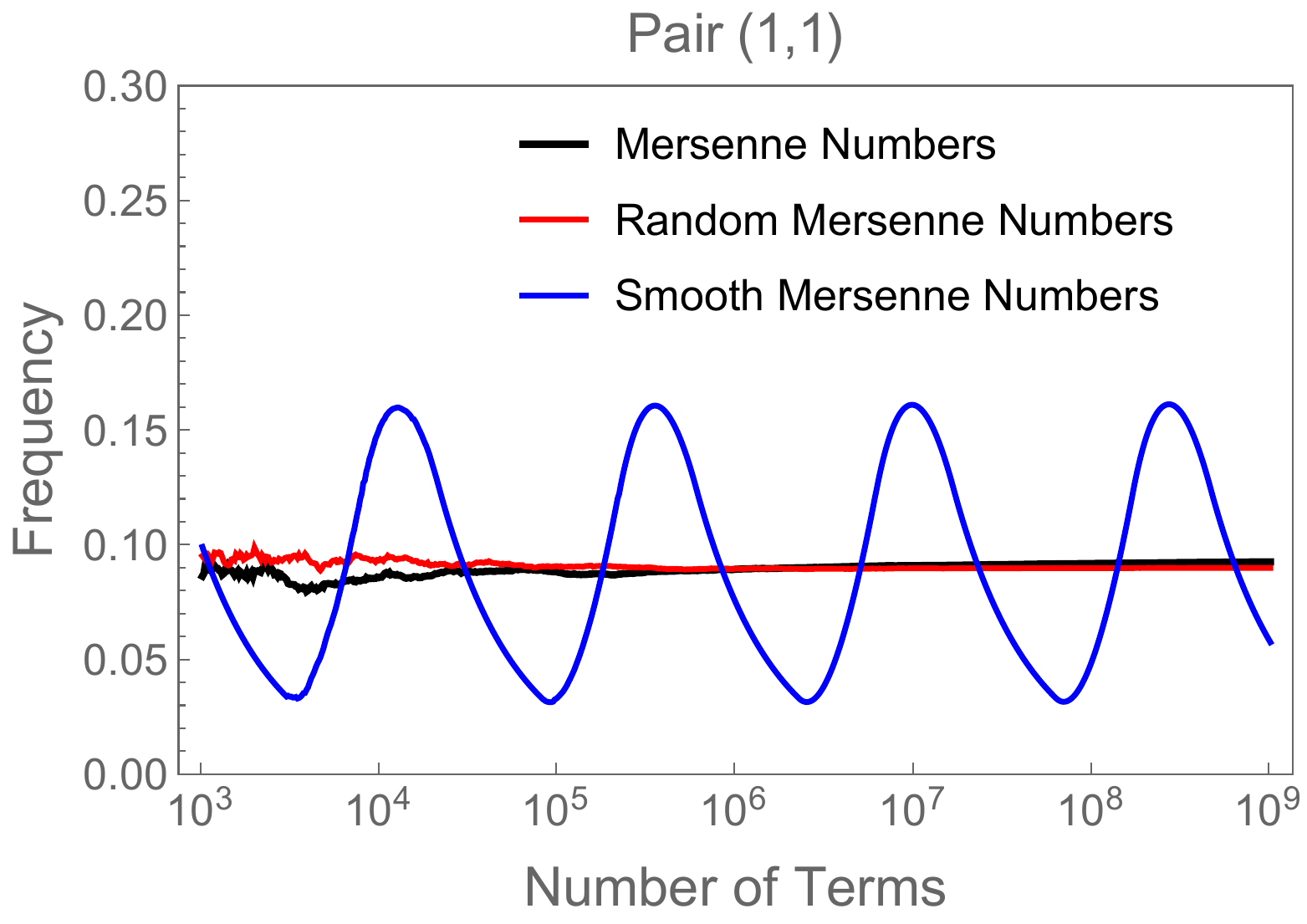}
\hspace{1em}
\includegraphics[width=.45\textwidth]{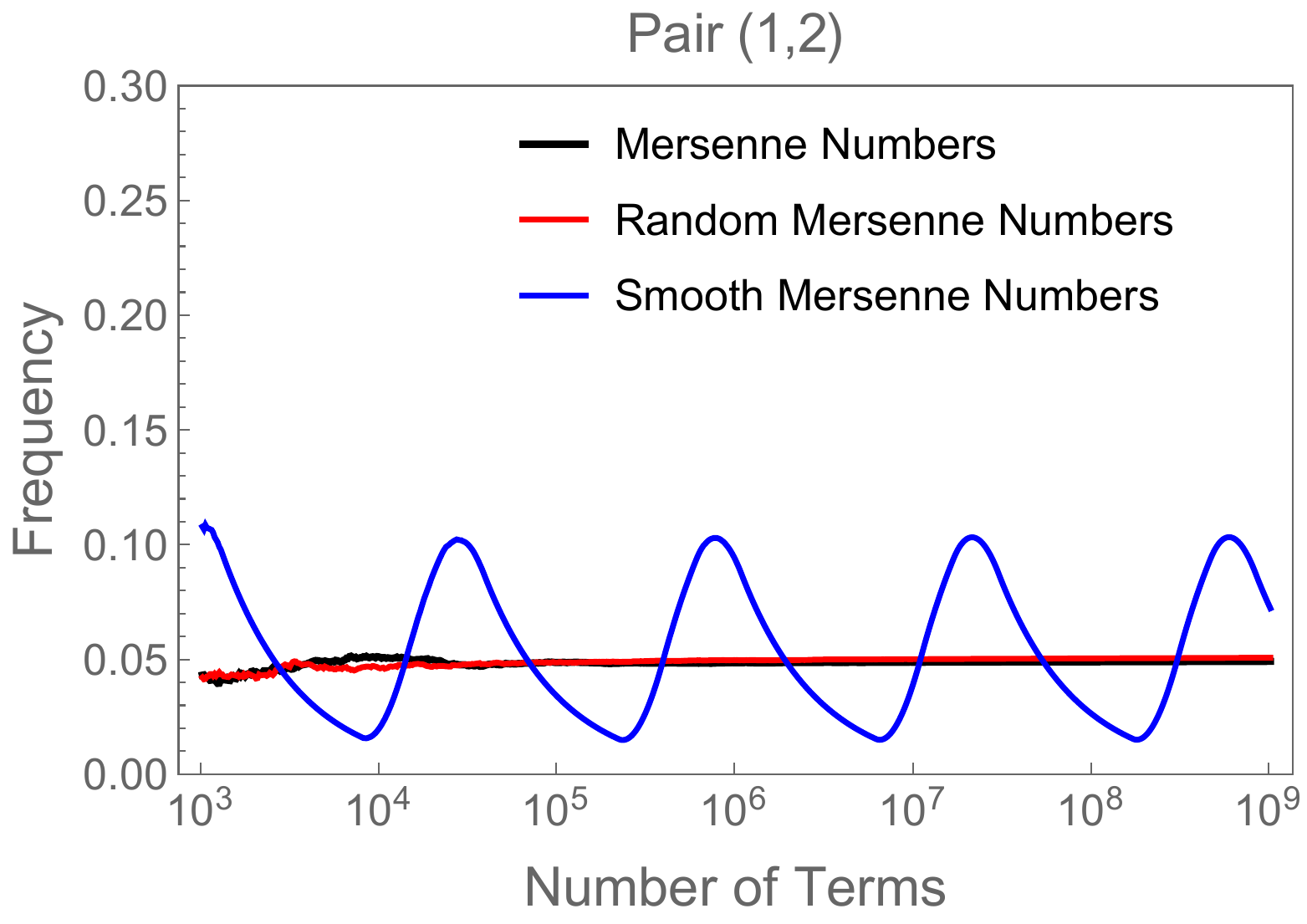}
\\[2ex]
\includegraphics[width=.45\textwidth]{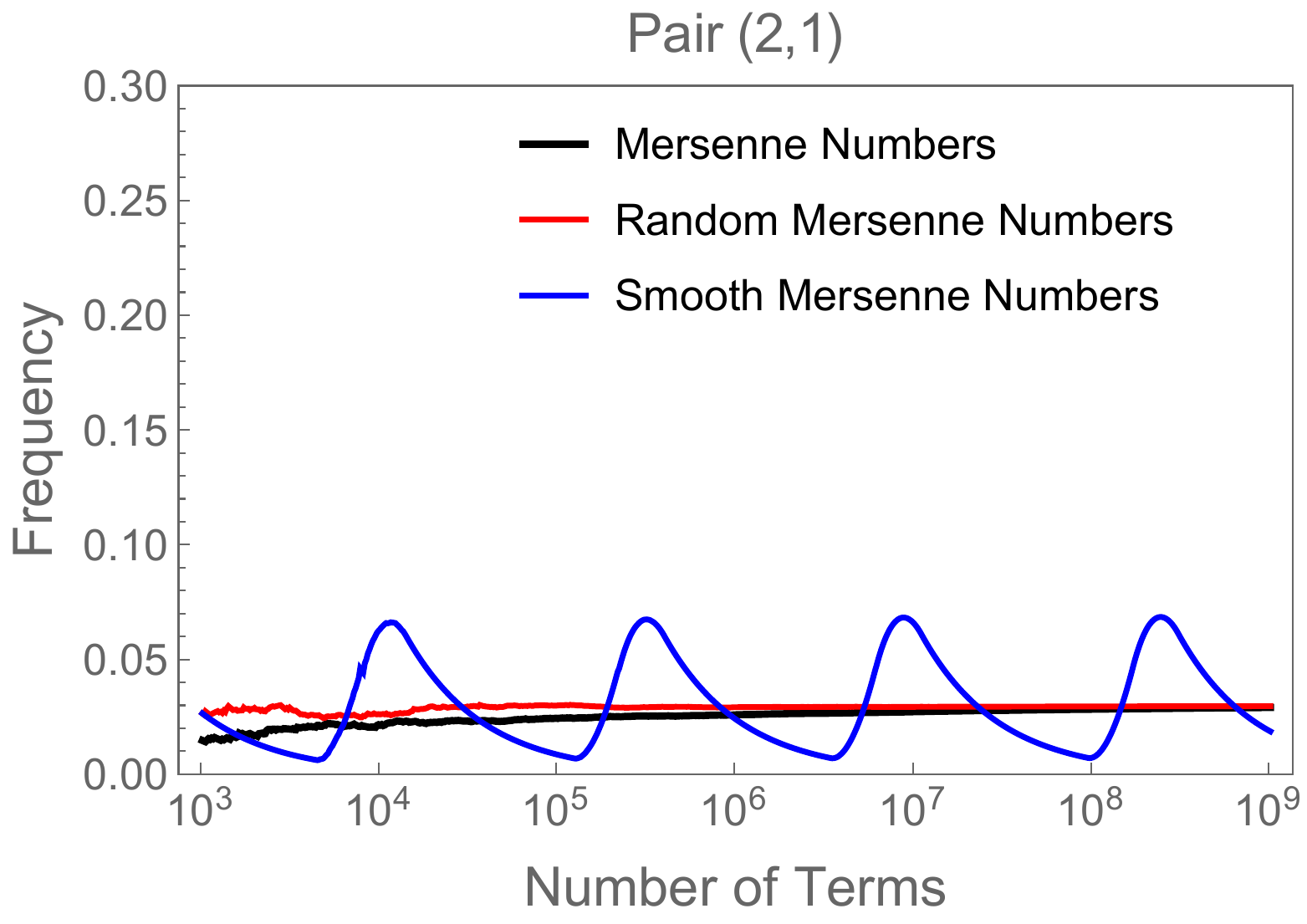}
\hspace{1em}
\includegraphics[width=.45\textwidth]{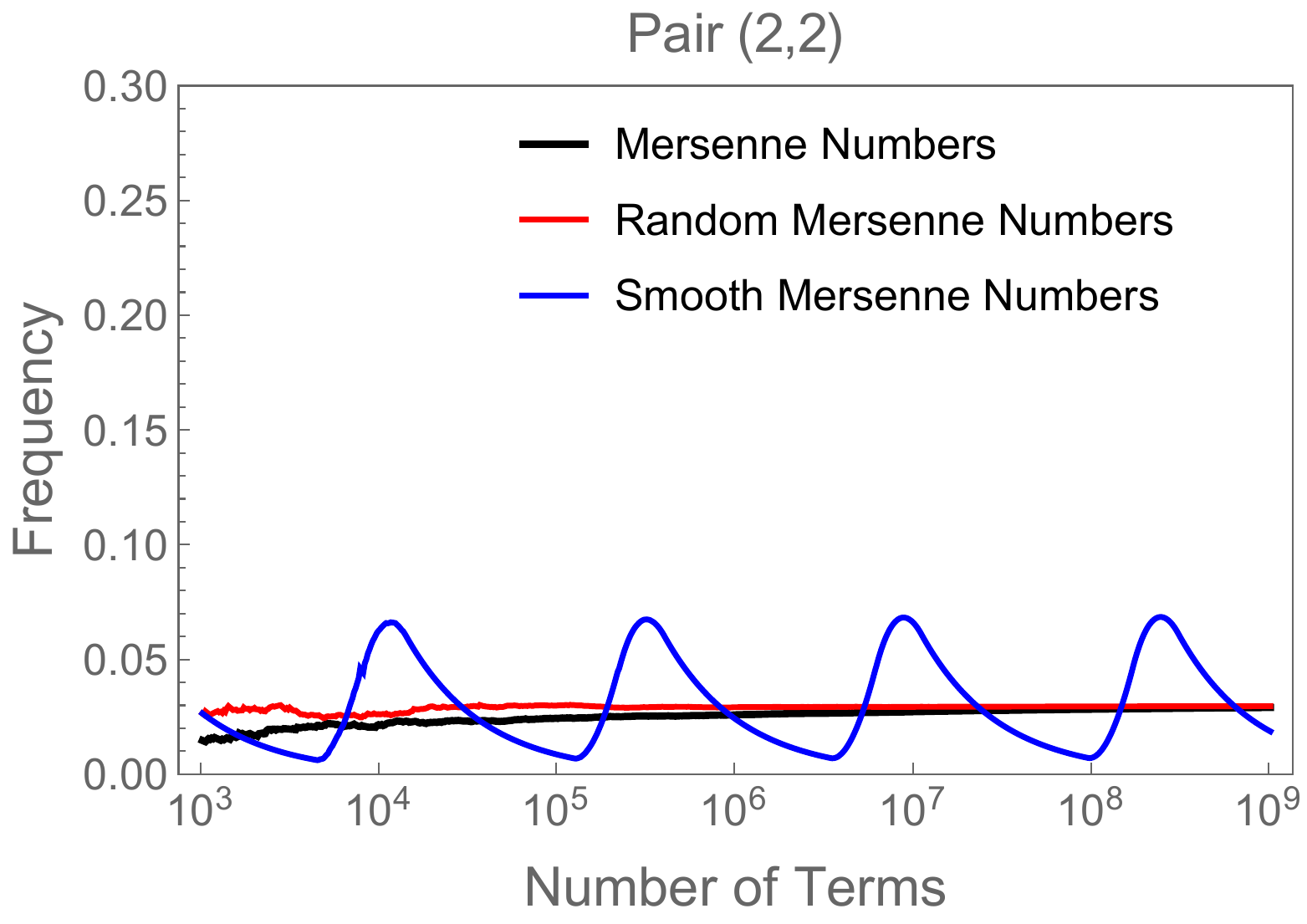}
\end{center}
\caption{%
The distribution of pairs $(d_1,d_2)$ of leading digits of successive
terms in the sequences of Mersenne numbers, random Mersenne numbers, 
and smooth Mersenne numbers.  The behavior of the latter sequence is
distinctly different from that of the former two sequences.  For the  
Mersenne and random Mersenne sequences the frequencies appear to converge
to the expected limit, given in the first column of Table
\ref{table:benford-pairs-combined},
while for the smooth Mersenne
numbers these frequencies exhibit an oscillating behavior.
}
\label{fig:benford-pairs-combined} 
\end{figure}


For additional insight and support for Conjecture
\ref{conj:local-benford}, we consider the \emph{variation distance}
between the observed and predicted pair distributions.  The variation
distance (or \emph{total variation distance}) is a standard
distance measure for probability distributions.
Given discrete probability distributions $P$ and $Q$ on a 
(finite or countable) 
probability space $\Omega$, the variation distance between $P$ and $Q$ 
is defined as 
\begin{equation}
\label{eq:variation-distance-abstract}
\tv(P,Q)=
\frac12\sum_{\omega\in\Omega}|P(\omega)-Q(\omega)|.
\end{equation}
In the case $\Omega$ is the set of $k$-tuples $(d_1,\dots,d_k)$,
$d_i=1,2,\dots,9$, 
and $P$ the predicted distribution on this set given by
\eqref{eq:k-tuples-expected}, this definition reduces to 
\begin{equation}
\label{eq:variation-distance-k-tuples}
\tv(P,Q)=
\frac12\sum_{d_1,\dots,d_k=1}^9\left|
P(d_1)\dots P(d_k)-Q(d_1,\dots,d_k)
\right|,
\end{equation}
where $P(d)=\log_{10}(1+1/d)$ are the individual Benford frequencies for
digit $d$.

Conjecture \ref{conj:local-benford} can be restated in terms of the 
variation distance $\tv$: The conjecture is equivalent to the statement 
$\tv(Q_N,P)\to 0$ as $N\to\infty$,
where $P$ is the predicted distribution given by
\eqref{eq:k-tuples-expected} and 
$Q_N=Q_N(d_1,\dots,d_k)$ denotes the \emph{observed} frequency
of the tuple of leading digits $(d_1,\dots,d_k)$ among the first $N$
Mersenne numbers.

Figure \ref{fig:benford-variation-distance} shows the behavior of this
variation distance for $k=1$ and $k=2$ for the sequences of Mersenne
numbers, random Mersenne numbers, and smooth Mersenne numbers.  For $k=1$
the behavior is essentially the same for all three sequences: In all
cases, the variation distance clearly converges to $0$. This is
consistent with the global Benford distribution properties of these
sequences discussed earlier.

For the case $k=2$, however, significant differences emerge. The most noticeable
difference is that, for the sequence of smooth Mersenne numbers, 
the variation distance does not decay as the number of terms increases,
but oscillates between values of around $0.27$ and $0.3$. By contrast, for
the Mersenne and random Mersenne sequences, the variation distance
decreases as the number of terms increases and appears to converge slowly
to $0$.

\begin{figure}[H]
\begin{center}
\includegraphics[width=.45\textwidth]{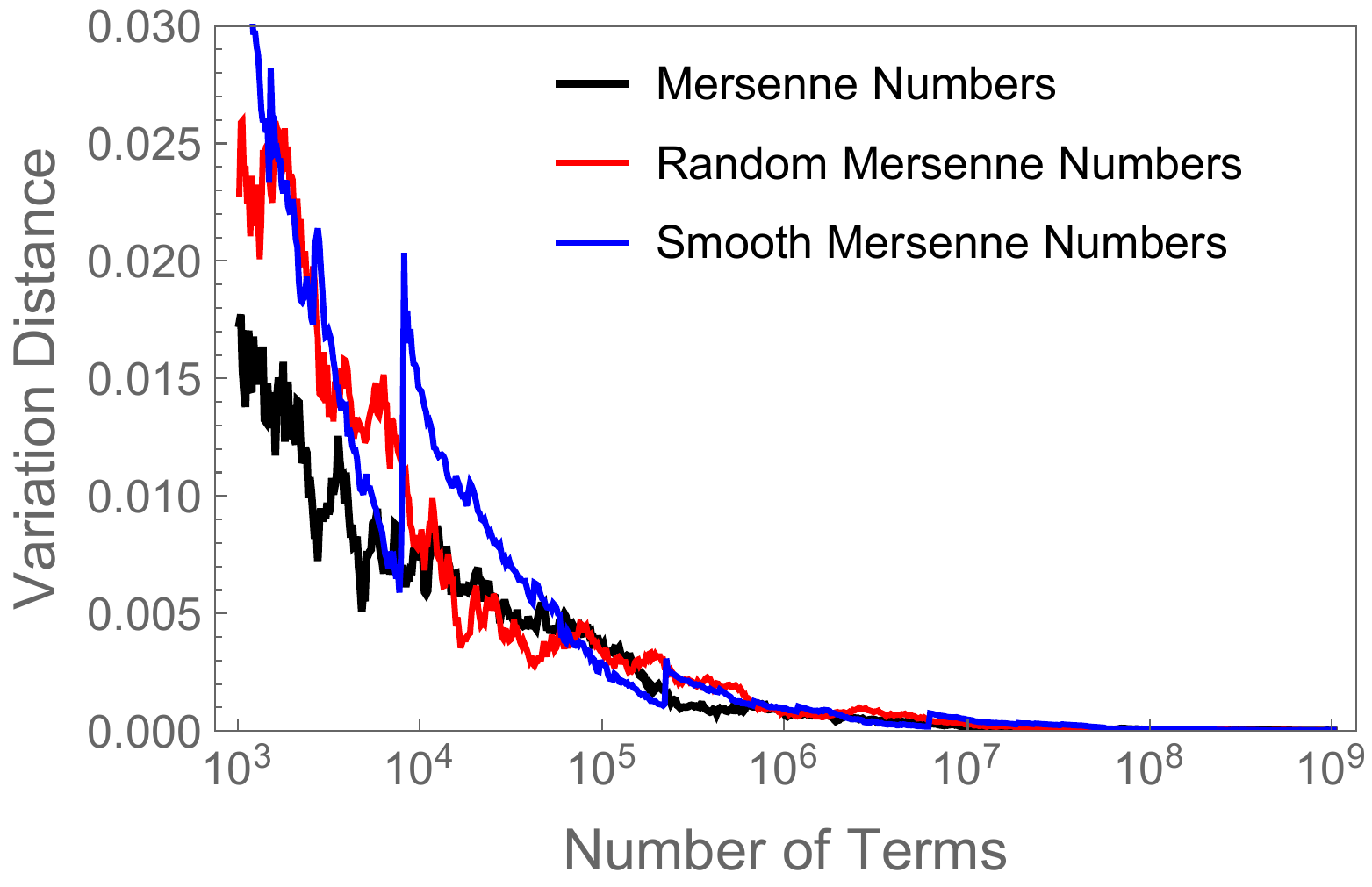}
\hspace{1em}
\includegraphics[width=.45\textwidth]{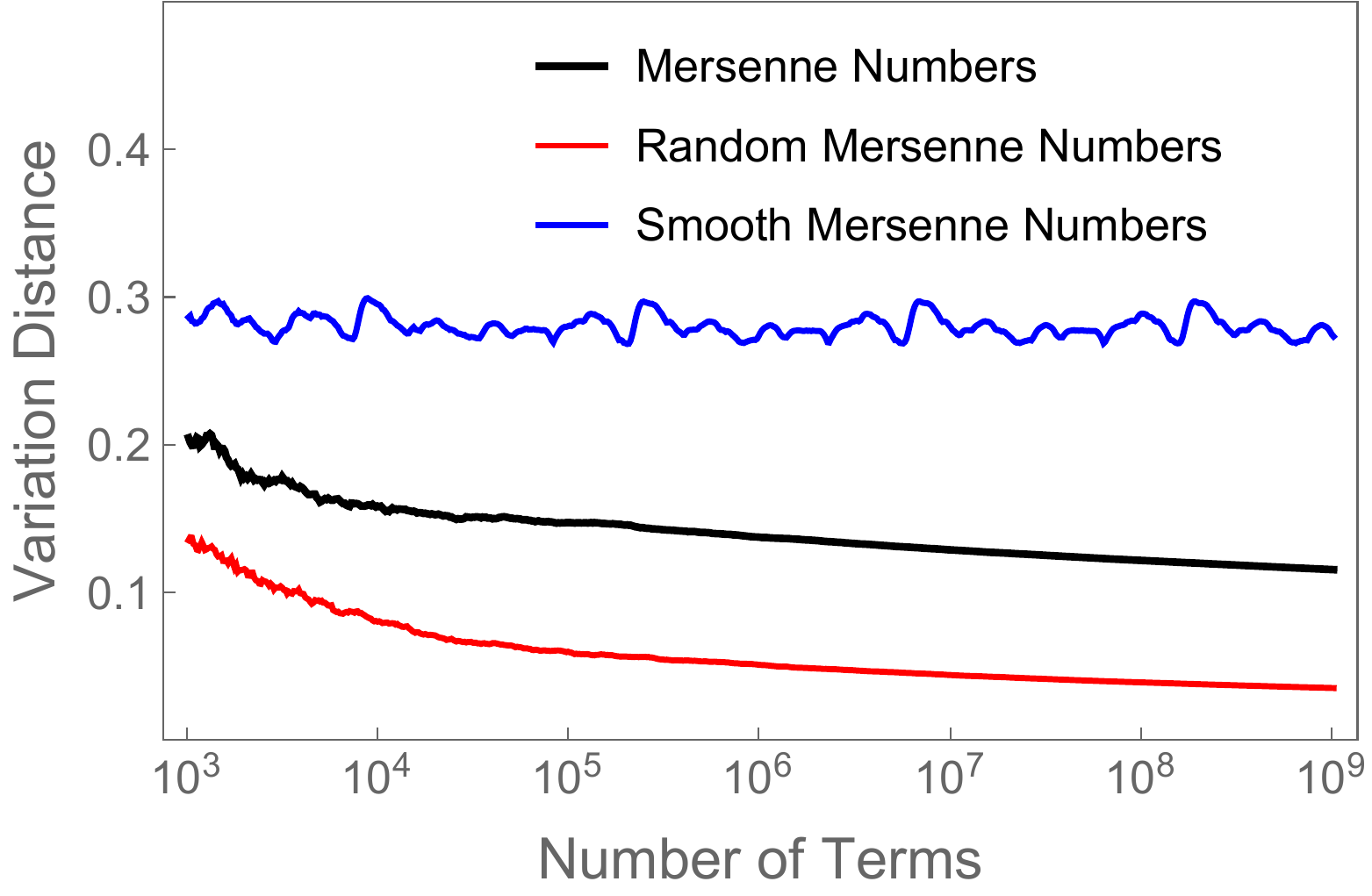}
\end{center}
\caption{%
Variation distance between observed and predicted frequencies of
$k$-tuples of leading digits for $k=1$ (left figure) and $k=2$
(right figure) in the sequence of Mersenne numbers, random Mersenne
numbers, and smooth Mersenne numbers. When $k=1$, 
the variation distance approaches $0$ in all three cases.
When $k=2$, the variation distance appears to approach $0$
for the Mersenne and random Mersenne numbers, but not for the smooth
Mersenne numbers.
}
\label{fig:benford-variation-distance} 
\end{figure}


We conclude this section by commenting on the quality of the approximations 
in the \emph{local} distribution of leading digits of Mersenne numbers
predicted by Conjectures \ref{conj:local-benford} and
\ref{conj:benford-waits}.
As Figure \ref{fig:benford-variation-distance} shows, while
the variation distance for frequencies of  pairs of leading digits 
of Mersenne numbers and of random Mersenne numbers both seem to converge
to $0$, the rate of convergence is noticeably better for random
Mersenne numbers.  A similar difference in the quality of the fit can be
observed in the distribution of pairs of leading digits in Table
\ref{table:benford-pairs-combined}, and in the waiting time frequencies
shown in Figure \ref{fig:benford-waits-details}. 

A plausible explanation for this difference in the quality of fit between
Mersenne numbers and their random analogs is the slow rate of growth of
the differences between consecutive primes, $p_{n+1}-p_n$, along with
divisibility constraints of these differences.  These constraints become
negligible as $n\to\infty$, but they do have an influence when $n$ is
small.  For example, for $n\approx 10^9$, we expect the differences
$p_{n+1}-p_n$ to be of average size $\log(10^9)\approx 20$, 
and we expect differences between random primes to be of similar order of
magnitude.  However, for the prime numbers these differences 
must satisfy additional congruence constraints which significantly 
restrict the set of $k$-tuples of positive integers that may occur as
$k$-tuples of differences between consecutive primes.
For example, apart from finitely many exceptions, all elements in such a
$k$-tuple of consecutive prime differences must be even, and no two
consecutive elements can have the same non-zero remainder modulo $3$.
For random primes there are no such constraints other than the restriction
to integer values, thus causing random Mersenne numbers to have (initially)
better Benford distribution properties on a local scale.


\section{Proof of Theorem \ref{thm:smooth-local-benford}\protect
\footnote{%
This result could also be derived
from \cite[Theorem 2.8]{local-benford}, 
a general result on the local Benford distribution properties of a large 
class of arithmetic sequences. We present here an elementary
argument that does not depend on the theory of uniform distribution modulo $1$
and which suffices  to obtain the assertions of the theorem.}}
\label{sec:smooth-local}

Let $\{a_n\}$ be a sequence of positive real numbers satisfying 
the conditions \eqref{eq:smooth-local-benford-I}
and \eqref{eq:smooth-local-benford-II} of the theorem.
Setting 
\begin{equation*}
u_n=\Delta\log_{10} a_n=\log_{10} a_{n+1}-\log_{10} a_n,
\end{equation*}
these conditions can be written as 
\begin{equation}
\label{eq:smooth-local-benford-Ia}
u_n\to\infty \quad (n\to\infty)
\end{equation}
and 
\begin{equation}
\label{eq:smooth-local-benford-IIa}
u_{n+1}=u_n+O\left(\frac1n\right),
\end{equation}
respectively.\footnote{Note 
that the conditions \eqref{eq:smooth-local-benford-I} and
\eqref{eq:smooth-local-benford-II} of the theorem are independent of the
base of the logarithm chosen in the definition of the $\Delta$ operator. 
For our proof it is convenient to work with base
$10$ logarithms instead of natural logarithms, so we have defined $u_n$
in terms of the base $10$ logarithm.}

Given a positive integer $k$, 
let $n_k$ denote the smallest integer $n$ with $u_n\ge k+1/2$. Then 
\begin{equation}
\label{eq:un-k1}
u_{n_k-1}< k+\frac12\le u_{n_k}.
\end{equation}
By condition \eqref{eq:smooth-local-benford-Ia},
$n_k$ is well-defined provided $k$ is sufficiently large,  which we
shall henceforth assume.

Conditions \eqref{eq:smooth-local-benford-IIa} 
and \eqref{eq:un-k1} imply 
\begin{align*}
u_{n_k}&=k+\frac12+O\left(\frac1{n_k}\right)
\end{align*}
and 
\begin{align*}
u_n&=k+\frac12+O\left(\frac{n-n_k}{n_k}\right)
\quad (n_k\le n< 2n_k).
\end{align*}
It follows that, given $\epsilon>0$,  there exists $\delta=\delta(\epsilon)>0$ 
such that, for sufficiently large $k$,
\begin{equation}
\label{eq:un-k2}
\left|u_n-k-\frac12\right|\le \epsilon\quad (n_k\le n<(1+\delta)n_k).
\end{equation}

We now show that
\eqref{eq:un-k2} is incompatible with the assumption that 
$\{a_n\}$ is locally Benford distributed of order $2$ (or greater), 
or that $\{a_n\}$ has Benford-distributed waiting times.  
Indeed, either of these assumptions implies that for all sufficiently
large $k$ there exist integers $n$ with $n_k\le n<(1+\delta)n_k$ such that 
$D(a_n)=D(a_{n+1})=1$.  The latter condition is equivalent to  
\[
\{\log_{10}a_n\}\in[0,\log_{10}2)
\quad\text{and}\quad
\{\log_{10}a_{n+1}\}\in[0,\log_{10}2).
\]
Since $u_n=\log_{10}a_{n+1}-\log_{10}a_n$, 
it follows that if $D(a_n)=D(a_{n+1})=1$, then
\[
|u_n-m|\le \log_{10}2=0.301\dots \quad \text{for some $m\in\ZZ$.}
\]
This contradicts \eqref{eq:un-k2} upon  choosing $\epsilon=0.1$.  
This completes the proof of Theorem \ref{thm:smooth-local-benford}.

\section{Summary and Concluding Remarks}
\label{sec:conclusion}

In this paper we presented the results of a large scale numerical 
investigation of the distribution of leading digits of the Mersenne
numbers $M_n=2^{p_n}-1$, where $p_n$ is the $n$-th prime number.  Our main
empirical finding is that the leading digits of $\{M_n\}$ behave like a sequence
of independent Benford-distributed random variables, on both a
\emph{global} and a \emph{local} scale.  The observed \emph{local}
behavior is in stark contrast to the behavior exhibited by other
exponentially growing arithmetic sequences,
which typically satisfy Benford's Law on a \emph{global} scale, 
but tend to have very poor \emph{local} Benford distribution properties.

We have provided heuristic and numerical evidence suggesting that it is
the statistical \emph{irregularities} in the distribution of primes that
cause the leading digit distribution of the Mersenne numbers,
to be unusually \emph{regular}.  On the one hand,
replacing the prime numbers $p_n$ by their smooth approximations $n\log
n$ in the definition of $M_n$ yields a leading digit sequence that
behaves similarly on a \emph{global} scale, but has a completely
different, and highly irregular, \emph{local} behavior.  On the other
hand, replacing the prime numbers $p_n$ by appropriately defined ``random
primes'' $p_n^*$ yields a behavior similar to that displayed by the
Mersenne sequence.  

This heuristic explanation for the ``unreasonably'' good fit of Benford's
Law to the sequence of Mersenne numbers is consistent with probabilistic
explanations of Benford's Law in terms of random processes such as
repeated multiplications of random quantities; see, for example, Berger
and Hill \cite{berger2011}, Miller and Nigrini \cite{miller-nigrini2008},
and Chenavier et al. \cite{chenavier2017}.

\bigskip

Our random model for Mersenne numbers is based on the classic Cram\'er
model for the distribution of primes, in which the events ``$n$ is
prime'', $n=3,4,\dots$,  are independent events, occurring with
probability $1/\log n$.  This essentially says that primes occur
according to a Poisson process with interarrival times increasing at a
rate $\log n$. Whether the actual sequence of primes behaves in this way
is still conjectural, but Gallagher \cite{gallagher1976} (see also 
Soundararajan \cite{sound2007}) showed that, 
under the assumption of a generalized prime $k$-tuples conjecture, this
is indeed the case.
The close agreement between the leading digit behavior of the Mersenne
numbers and that of a random analog of these numbers in which
the primes are replaced by random primes can be seen as further evidence
that the primes indeed behave according to a Poisson process.

We note that, for certain questions,  Cram\'er's model does not give
the correct prediction for the behavior of primes. In particular, Maier
\cite{maier1985} showed that the \emph{maximal} gaps between prime
numbers are significantly larger than those predicted by the model. 
However, these results do not affect the conjectured \emph{distribution}
of gaps based on the Poisson model; see the remark before Section 1.1 in
Soundararajan \cite{sound2007}.  

\bigskip

To conclude this section, we comment on possible extensions and
generalizations of the results and conjectures we have presented, and
some open questions suggested by these results.
The most obvious extension is to leading digits with respect to other
bases.  We expect all of our results and conjectures to remain valid for
leading digits in a general base $b$, provided one excludes trivial
situations such as bases that are powers of $2$.

Another natural extension is to sequences of the form $\{a^{p_n}\}$. 
Excluding trivial situations, we expect leading digits in these sequences
to behave like those of the Mersenne numbers, $2^{p_n}-1$. 
For example, the proof of Theorem \ref{thm:Mn-global} shows that any 
sequence $\{a^{p_n}\}$ for which $\log_{10} a$ is irrational satisfies
Benford's Law. Similarly, we  expect sequences of this form to have the same 
\emph{local} Benford distribution properties as the sequence of Mersenne
numbers. 

Interestingly, this does not hold for the sequence of primorial numbers,
$P_n=\prod_{k=1}^n p_k$, which have a similar rate of growth as the
Mersenne numbers, $2^{p_n}-1$, but are more ``smooth'' at a local level.
For example,  we have
$\log M_{n+1}-\log M_n\sim (\log 2)(p_{n+1}-p_n$), while $\log P_{n+1}-\log P_n
= \log p_{n+1}$.  Indeed, using an argument similar to that in the proof
of Theorem \ref{thm:smooth-local-benford}, one can show that the sequence
$\{P_n\}$ is \emph{not} locally Benford distributed of order $2$ or
larger.

The numerical data we presented in support of Conjecture
\ref{conj:benford-error} suggests the possibility that the error in
the Benford approximation to leading digit counts of Mersenne numbers
may be asymptotically normally distributed.  This would represent a
considerable strengthening of Conjecture \ref{conj:benford-error}.

A natural question is whether any of the conjectured results about the
Mersenne numbers can be proved rigorously. Given our current state of
knowledge on the distribution of primes, unconditional results seem
unlikely; however, by following the method of Gallagher \cite{gallagher1976}
it may be possible to prove Conjectures \ref{conj:local-benford}
and \ref{conj:benford-waits} conditionally,  assuming an
appropriate version of the prime $k$-tuples conjecture.


In this paper we focused on the distribution of the \emph{leading} (i.e.,
\emph{leftmost}) digit of $M_n$. One can ask more generally for the distribution
of the $k$-th digit in the base $10$ (or base $b$) expansion of $M_n$.
For fixed $k$, we expect similar results to hold, with the Benford
distribution replaced by an appropriate analog for the $k$-th leading
digit. More interesting is the case when $k$ is
an increasing  function of $n$. In this case, we expect the Benford
phenomenon to gradually fade away as $k$ gets larger and the distribution to
approach a uniform distribution. For example, it seems plausible that the
distribution of the \emph{middle} digits of the sequence of Mersenne
numbers is asymptotically uniform, though results of this type appear to
be out of reach given our current state of knowledge.   On the other hand,
the distribution of the \emph{rightmost} digit (and, more generally, the
$k$-th digit from the right) can be studied using exponential sums, and
some results are known for the sequence of Mersenne numbers and similar
exponentially growing sequences; see Banks et al. \cite{banks2012} and the
references therein.

\subsection*{Acknowledgements.}  
We thank the referees for a careful reading of the paper and a number of 
comments that helped improve the paper.


\providecommand{\bysame}{\leavevmode\hbox to3em{\hrulefill}\thinspace}
\providecommand{\MR}{\relax\ifhmode\unskip\space\fi MR }
\providecommand{\MRhref}[2]{%
  \href{http://www.ams.org/mathscinet-getitem?mr=#1}{#2}
}
\providecommand{\href}[2]{#2}



\begin{thebibliography}{10}

\bibitem{apostol1976}
Tom~M. Apostol, \emph{Introduction to analytic number theory}, Springer-Verlag,
  New York-Heidelberg, 1976, Undergraduate Texts in Mathematics. \MR{0434929}


\bibitem{banks2009}
William~D. Banks and Igor~E.
  Shparlinski, \emph{Prime numbers with Beatty sequences}, 
  Colloq. Math. \textbf{115} (2009), no.~2, 147--157.. \MR{2491740}


\bibitem{banks2012}
William~D. Banks, John~B. Friedlander, Moubariz~Z. Garaev, and Igor~E.
  Shparlinski, \emph{Exponential and character sums with {M}ersenne numbers},
  J. Aust. Math. Soc. \textbf{92} (2012), no.~1, 1--13. \MR{2945673}

\bibitem{bateman-diamond2004}
Paul~T. Bateman and Harold~G. Diamond, \emph{Analytic number theory},
  Monographs in Number Theory, vol.~1, World Scientific Publishing Co. Pte.
  Ltd., Hackensack, NJ, 2004, An introductory course. \MR{2111739}

\bibitem{beebe2017}
Nelson H.~F. Beebe, \emph{A bibliography of publications about {B}enford's law,
  {H}eap's law and {Z}ipf's law},
  \url{ftp://ftp.math.utah.edu/public_html/public_html/pub/tex/bib/benfords-law.pdf},
  2017.

\bibitem{benford1938}
Frank Benford, \emph{The law of anomalous numbers}, Proceedings of the American
  Philosophical Society \textbf{78} (1938), no.~4, 551--572.

\bibitem{berger2011}
Arno Berger and Theodore~P. Hill, \emph{A basic theory of {B}enford's law},
  Probab. Surv. \textbf{8} (2011), 1--126. \MR{2846899}

\bibitem{berger2015}
\bysame, \emph{An introduction to {B}enford's law}, Princeton University Press,
  Princeton, NJ, 2015. \MR{3242822}

\bibitem{benfordonline}
Arno Berger, Theodore~P. Hill, and E.~Rogers, \emph{Benford online
  bibliography}, \url{http://www.benfordonline.net}, Last accessed 11.04.2017.

\bibitem{benford-error}
Zhaodong Cai, Matthew Faust, 
A.~J. Hildebrand, Junxian Li, and Yuan Zhang, \emph{The unreasonable
effectiveness of Benford's Law in mathematics}, Preprint.

\bibitem{local-benford}
Zhaodong Cai, A.~J. Hildebrand, and Junxian Li, \emph{A local {B}enford {L}aw
  for a class of arithmetic sequences}, Preprint,
  \url{http://arXiv.org/1808.01496}.


\bibitem{chenavier2017}
Nicolas Chenavier, Bruno Mass{\'e}, and Dominique Schneider, \emph{Products of
  random variables and the first digit phenomenon}, Stochastic Processes and
  their Applications (2017), in press.

\bibitem{cohen-katz1984}
Daniel I.~A. Cohen and Talbot~M. Katz, \emph{Prime numbers and the first digit
  phenomenon}, J. Number Theory \textbf{18} (1984), no.~3, 261--268.
  \MR{746863}

\bibitem{diaconis1977}
Persi Diaconis, \emph{The distribution of leading digits and uniform
  distribution {${\rm mod}$} {$1$}}, Ann. Probability \textbf{5} (1977), no.~1,
  72--81. \MR{0422186}

\bibitem{eliahou2013}
S.~Eliahou, B.~Mass\'e, and D.~Schneider, \emph{On the mantissa distribution of
  powers of natural and prime numbers}, Acta Math. Hungar. \textbf{139} (2013),
  no.~1-2, 49--63. \MR{3028653}

\bibitem{fuchs-letta1996}
Aim\'e Fuchs and Giorgio Letta, \emph{Le probl\`eme du premier chiffre
  d\'ecimal pour les nombres premiers}, Electron. J. Combin. \textbf{3} (1996),
  no.~2, Research Paper 25, The Foata Festschrift. \MR{1392510}

\bibitem{gallagher1976}
Patrick~X. Gallagher, \emph{On the distribution of primes in short intervals},
  Mathematika \textbf{23} (1976), no.~1, 4--9. \MR{0409385}

\bibitem{hill1995}
Theodore~P. Hill, \emph{The significant-digit phenomenon}, Amer. Math. Monthly
  \textbf{102} (1995), no.~4, 322--327. \MR{1328015}

\bibitem{hurliman2009}
Werner H{\"u}rlimann, \emph{Generalizing {B}enford's law using power laws:
  application to integer sequences}, International Journal of Mathematics and
  Mathematical Sciences \textbf{2009} (2009), 45. \MR{2533550}

\bibitem{hurliman2006}
Werner H\"urlimann, \emph{Benford's law from 1881 to 2006: a bibliography},
  \url{http://arxiv.org/ftp/math/papers/0607/0607168.pdf}, Last accessed
  11.04.2017.

\bibitem{iwaniec2004}
Henryk Iwaniec and Emmanuel Kowalski, \emph{Analytic number theory}, American
  Mathematical Society Colloquium Publications, vol.~53, American Mathematical
  Society, Providence, RI, 2004. \MR{2061214}

\bibitem{kuipers1974}
Lauwerens Kuipers and Harald Niederreiter, \emph{Uniform distribution of
  sequences}, Wiley-Interscience [John Wiley \& Sons], New York-London-Sydney,
  1974, Pure and Applied Mathematics. \MR{0419394}

\bibitem{luque-lacasa2009}
Bartolo Luque and Lucas Lacasa, \emph{The first-digit frequencies of prime
  numbers and {R}iemann zeta zeros}, Proc. R. Soc. Lond. Ser. A Math. Phys.
  Eng. Sci. \textbf{465} (2009), no.~2107, 2197--2216. \MR{2515637}

\bibitem{maier1985}
Helmut Maier, \emph{Primes in short intervals}, Michigan Math. J. \textbf{32}
  (1985), no.~2, 221--225. \MR{783576}

\bibitem{masse2011}
Bruno Mass{\'e} and Dominique Schneider, \emph{A survey on weighted densities
  and their connection with the first digit phenomenon}, Rocky Mountain J.
  Math. \textbf{41} (2011), no.~5, 1395--1415. \MR{2838069}

\bibitem{masse2014}
Bruno Mass\'e and Dominique Schneider, \emph{The mantissa distribution of the
  primorial numbers}, Acta Arith. \textbf{163} (2014), no.~1, 45--58.
  \MR{3194056}

\bibitem{masse2015}
Bruno Mass{\'e} and Dominique Schneider, \emph{Fast growing sequences of
  numbers and the first digit phenomenon}, International Journal of Number
  Theory \textbf{11} (2015), no.~705, 705--719. \MR{3327839}

\bibitem{miller2015}
Steven~J. Miller (ed.), \emph{Benford's {L}aw: theory and applications},
  Princeton University Press, Princeton, NJ, 2015. \MR{3408774}

\bibitem{miller-nigrini2008}
Steven~J. Miller and Mark~J. Nigrini, \emph{The modulo 1 central limit theorem
  and {B}enford's law for products}, Int. J. Algebra \textbf{2} (2008),
  no.~1-4, 119--130. \MR{2417189}

\bibitem{muller2001irram}
Norbert M{\"u}ller, \emph{The {iRRAM}: {E}xact arithmetic in {C++}},
  Computability and Complexity in Analysis, Springer, 2001, pp.~222--252.

\bibitem{newcomb1881}
Simon Newcomb, \emph{Note on the frequency of use of the different digits in
  natural numbers}, Amer. J. Math. \textbf{4} (1881), no.~1-4, 39--40.
  \MR{1505286}

\bibitem{nigrini2012}
Mark Nigrini, \emph{Benford's law: Applications for forensic accounting,
  auditing, and fraud detection}, vol. 586, John Wiley \& Sons, 2012.

\bibitem{raimi1976}
Ralph~A. Raimi, \emph{The first digit problem}, Amer. Math. Monthly \textbf{83}
  (1976), no.~7, 521--538. \MR{0410850}

\bibitem{schatte1983}
Peter Schatte, \emph{On {$H\sb{\infty }$}-summability and the uniform
  distribution of sequences}, Math. Nachr. \textbf{113} (1983), 237--243.
  \MR{725491}

\bibitem{sound2007}
Kannan Soundararajan, \emph{The distribution of prime numbers},
  Equidistribution in number theory, an introduction, NATO Sci. Ser. II Math.
  Phys. Chem., vol. 237, Springer, Dordrecht, 2007, pp.~59--83. \MR{2290494}

\bibitem{vinogradov1947}
I.~M. Vinogradov, \emph{The method of trigonometrical sums in the theory of
  numbers}, Trav. Inst. Math. Stekloff \textbf{23} (1947), 109. \MR{0029417}

\bibitem{whitney1972}
Raymond~E. Whitney, \emph{Initial digits for the sequence of primes}, Amer.
  Math. Monthly \textbf{79} (1972), 150--152. \MR{0304337}

\end{thebibliography}

\end{document}